\documentclass[notitlepage, 11pt]{scrartcl}
\usepackage[shortlabels]{enumitem}
\usepackage{amsmath}
\usepackage{amssymb}
\usepackage{amsthm}
\usepackage{abstract}
\usepackage{xcolor}
\usepackage{comment}
\usepackage{mathtools}
\usepackage{graphicx}
\usepackage{caption}
\usepackage{subcaption}
\usepackage{float}
\graphicspath{ {images_for_expansion/} }
\usepackage{hyperref}
\hypersetup{
    colorlinks=true,
    linkcolor=blue,
    filecolor=magenta,      
    urlcolor=cyan
    }
\makeatletter
\newcommand*{\barfix}[2][.175ex]{%
  \mathpalette{\@barfix{#1}}{#2}%
}
\newcommand*{\@barfix}[3]{%
  \vbox{%
    \kern#1\relax
    \hbox{$#2#3\m@th$}%
  }%
}
\makeatother

\newtheorem{theorem}{Theorem}
\newtheorem{thm}{Theorem}[section]
\newtheorem{corollary}[thm]{Corollary}
\newtheorem{lemma}[thm]{Lemma}

\newtheorem{remark}[thm]{Remark}

\newtheorem{question}[thm]{Question}
\newcommand{\footremember}[2]{%
    \footnote{#2}
    \newcounter{#1}
    \setcounter{#1}{\value{footnote}}%
}

\usepackage[margin=1in]{geometry}

\title{\vspace{-1.5cm}Isoperimetric Inequalities and Supercritical Percolation on High-dimensional Graphs} 
\author{%
Sahar Diskin \footremember{alley}{\scriptsize{School of Mathematical Sciences, Tel Aviv University, Tel Aviv 6997801, Israel. Email: sahardiskin@mail.tau.ac.il.}}%
\and Joshua Erde \footremember{trailer}{\scriptsize{Institute of Discrete Mathematics, Graz University of Technology, Steyrergasse 30, 8010 Graz, Austria. Email: erde@math.tugraz.at.}}%
\and Mihyun Kang \footremember{alley2}{\scriptsize{Institute of Discrete Mathematics, Graz University of Technology, Steyrergasse 30, 8010 Graz, Austria. Email: kang@math.tugraz.at.}}%
\and Michael Krivelevich \footremember{trailer2}{\scriptsize{School of Mathematical Sciences, Tel Aviv University, Tel Aviv 6997801, Israel. Email: krivelev@tauex.tau.ac.il.}}%
}
\begin{document}
\maketitle
\vspace{-2em}
\begin{abstract}
It is known that many different types of finite random subgraph models undergo quantitatively similar phase transitions around their percolation thresholds, and the proofs of these results rely on isoperimetric properties of the underlying host graph. Recently, the authors showed that such a phase transition occurs in a large class of regular high-dimensional product graphs, generalising a classic result for the hypercube.

In this paper we give new isoperimetric inequalities for such regular high-dimensional product graphs, which generalise the well-known isoperimetric inequality of Harper for the hypercube, and are asymptotically sharp for a wide range of set sizes. We then use these isoperimetric properties to investigate the structure of the giant component $L_1$ in supercritical percolation on these product graphs, that is, when $p=\frac{1+\epsilon}{d}$, where $d$ is the degree of the product graph and $\epsilon>0$ is a small enough constant.

We show that typically $L_1$ has edge-expansion $\Omega\left(\frac{1}{d\ln d}\right)$. Furthermore, we show that $L_1$ likely contains a linear-sized subgraph with vertex-expansion $\Omega\left(\frac{1}{d\ln d}\right)$. These results are best possible up to the logarithmic factor in $d$. 

Using these likely expansion properties, we determine, up to small polylogarithmic factors in $d$, the likely diameter of $L_1$ as well as the typical mixing time of a lazy random walk on $L_1$. Furthermore, we show the likely existence of a cycle of length $\Omega\left(\frac{n}{d\ln d}\right)$. These results not only generalise, but also improve substantially upon the known bounds in the case of the hypercube, where in particular the likely diameter and typical mixing time of $L_1$ were previously only known to be polynomial in $d$.
\end{abstract}

\section{Introduction}
\subsection{Background and motivation} 
In this paper we investigate the typical structure of the largest component after supercritical percolation in a certain class of high-dimensional graphs. Of particular interest, both in their own right, but also as a tool to study other structural properties, are the \emph{isoperimetric properties} of the largest component, which have proven to be key to understanding the large-scale structure of the giant component in many percolation models. Unsurprisingly, in order to understand the likely isoperimetric properties of the giant component, it is first essential to study the isoperimetric properties of the host graph.

Very generally, for any space which is endowed with a notion of volume and boundary, the \emph{isoperimetric problem} is to determine which sets of fixed volume have the smallest boundary. In the case of graphs, a natural notion of boundary to consider is the \textit{edge-boundary}. Given a graph $G=(V,E)$ and a subset of the vertices $S\subseteq V(G)$, we write $\partial(S)$ for the edge-boundary of $S$, that is, the set of edges with one endpoint in $S$ and one endpoint in $V(G)\setminus S$. The isoperimetric problem is then equivalent to determining, for each $k \in \mathbb{N}$, the parameter
\[
i_k(G) \coloneqq \min_{S\subseteq V(G), |S|=k}\left\{\frac{|\partial(S)|}{k}\right\},
\]
and characterising the sets which achieve this minimum.
Of particular interest is the \emph{edge-isoperimetric constant} of $G$, given by
\begin{align*}
    i(G)\coloneqq \min_{k\le |V(G)|/2} \{i_k(G)\}.
\end{align*}
This is also sometimes called the \emph{Cheeger constant}, as it can be viewed as a discrete analogue of the Cheeger isoperimetric constant of a compact Riemannian manifold \cite{C70}. It turns out that the Cheeger constant is a fundamental graph parameter, and can be used to demonstrate deep links between the combinatorial, geometric, spectral and stochastic properties of graphs. For this reason \textit{expander graphs}, roughly speaking graphs whose Cheeger constant is bounded from below by an absolute constant, have turned out to be very important in diverse areas of discrete mathematics and computer science. We refer the reader to \cite{HLW06} for a comprehensive survey on expander graphs and their application.

Whilst in general it is NP-hard to determine even the edge-isoperimetric constant of an arbitrary graph \cite{GJS74}, much is known about the isoperimetric properties of particularly well-structured graph classes. In particular, a classical result of Harper solves the isoperimetric problem on the $d$-\textit{dimensional (binary) hypercube} $Q^d$, whose vertex set is $\{0,1\}^d$, and in which two vertices are adjacent if and only if their Hamming distance is one. Harper's result implies the following isoperimetric inequality:
\begin{thm}[\cite{H64}, see also \cite{L64, B67, H76}]\label{th: Harper}
Let $d \in \mathbb{N}$. For every $k\in\left[2^d\right]$
\[
i_k\left(Q^d\right) \geq d-\log_2 k.
\]
Furthermore, the only sets which achieve equality in the above estimate are subcubes.
\end{thm}
The isoperimetric problem has also been solved, at least asymptotically, in many other classes of lattice-like graphs, such as grids \cite{BI91,AB95}, Cartesian powers of graphs \cite{C02,BE03}, and Abelian Cayley graphs \cite{L15,BE18, BEKR22}. For further background, we refer the reader to the surveys \cite{B94, B99, H04} on discrete isoperimetric problems.

On the other hand, the isoperimetric properties of particularly `unstructured' graphs, that is, graphs without any clear geometric structure, have also been well-studied. It is known that Erd\H{o}s-R\'enyi (binomial) random graphs \cite{J81,FK16} and random $d$-regular graphs \cite{B88} have typically good expansion properties, and one can view the well-known Expander Mixing Lemma, due to Alon and Chung \cite{AC88}, as a bound on the edge-isoperimetric constant of pseudo-random $(n,d,\lambda)$-graphs (see also \cite{AM85}). Furthermore, the isoperimetric properties of such graphs have been a key tool in the study of their structural properties. 

In this paper, we consider a mixture of these two paradigms. We study properties of \textit{random subgraphs} of graphs coming from a family of graphs which are quite structured --- arising from \textit{high-dimensional products} of bounded graphs. As in other percolation models, it turns out that the isoperimetric properties of these random subgraphs are key to understanding their large-scale structure, and that in order to understand the likely isoperimetric properties in the percolated subgraphs, it is useful first to study the isoperimetric problem in the underlying product graphs.

Given a sequence of graphs $G^{(1)},\ldots, G^{(t)}$, the Cartesian product of $G^{(1)},\ldots, G^{(t)}$, denoted by $G=G^{(1)}\square \cdots \square G^{(t)}$ or $G=\square_{j=1}^{t}G^{(j)}$, is the graph with the vertex set
\begin{align*}
    V(G)=\left\{v=(v_1,v_2,\ldots,v_t) \colon v_j\in V(G^{(j)}) \text{ for all } j \in [t]\right\},
\end{align*}
and the edge set
\begin{align*}
   E(G)=\left\{uv \colon \begin{array}{l} \text{there is some } j\in [t] \text{ such that }  u_jv_j\in E\left(G^{(j)}\right)\\  \text{ and } u_m=v_m  
    \text{ for all } m \neq j  \end{array}\right\}.
\end{align*}
We call $G^{(j)}$ the \textit{base graphs of} $G$. Note that if each $G^{(j)}$ is $d_j$-regular, then $G$ is $d$-regular with $d \coloneqq \sum_{j=1}^t d_j$. Many well-studied families of graphs arise in this manner. For example, the $t$-dimensional hypercube $Q^t$ is the $t$-fold Cartesian product of a single edge. Other examples include tori, grids, Hamming graphs, and many examples of Cayley graphs of groups arising from direct products.

We will be interested in properties of random subgraphs of high-dimensional product graphs, that is, we consider \textit{bond percolation} on these graphs. Percolation theory was initiated in 1957 by Broadbent and Hammersley \cite{BH57} in order to model the flow of fluid through a medium with randomly blocked channels, and has become a major area of research. In (bond) percolation, given a \textit{host graph} $G$ and a probability $p\in[0,1]$, we form the random subgraph $G_p$ by including every edge of $G$ independently with probability $p$. Percolation has been studied extensively on various geometric `lattice-like' classes of graphs, and in particular on many of the families of graphs which arise naturally as high-dimensional product graphs such as high-dimensional hypercubes \cite{AKS81,BKL92}, tori \cite{HH07,HH11}, or Hamming graphs \cite{BCVSS05a,BCVSS05b} (see \cite[Chapter 13]{HH17} for a survey on many important results in these models). We refer the reader to the monographs \cite{K82, G99, BR06} for a more comprehensive background on percolation theory.

There is an intrinsic connection between the phase transition in percolated graphs, and the isoperimetric properties of the host graph. This connection can be seen, albeit implicitly, already in the classical phase transition result of Erd\H{o}s and R\'enyi \cite{ER60}. In the case of percolated expander graphs, this connection is explicit in the work of Alon, Benjamini and Stacey \cite{ABS04}, and in the case of percolated pseudo-random graphs in the work of Frieze, Krivelevich and Martin \cite{FKM04}. Ajtai, Koml\'os, and Szemer\'edi \cite{AKS81} proved that $Q^d_p$ undergoes a phase transition quantitatively similar to the one which occurs in $G(n,p)$, and their work was later extended by Bollob\'as, Kohayakawa, and \L{}uczak \cite{BKL92} --- both of which explicitly rely on the isoperimetric properties of the hypercube. 

Furthermore, above the percolation threshold the connection between the isoperimetric properties of the host graph $G$, the expansion properties of the percolated graph $G_p$, and the combinatorial properties of the resulting giant component in $G_p$ has been made explicit in several works. To mention a few, Fountoulakis and Reed \cite{FR07a, FR08}, and, independently, Benjamini, Kozma and Wormald \cite{BKW14} study the asymptotic mixing time of a random walk on the giant component of $G(n,p)$ using the likely expansion properties of connected sets (and, implicitly, the isoperimetric properties of the complete graph); Riordan and Wormald \cite{RW10} utilise likely expansion properties in the giant component of $G(n,p)$ in order to bound its typical diameter; and Erde, Kang and Krivelevich \cite{EKK22} use the isoperimetric properties of $Q^d$ to show typical expansion properties of the giant component of $Q^d_p$, and derive from them the current best known bounds on its likely circumference (that is, the length of a longest cycle), typical diameter and asymptotic mixing time.

Recently, generalising the results of \cite{AKS81, BKL92} on $Q^d$, the authors showed that any high-dimensional product graph, whose base graphs are bounded in order and regular, undergoes a phase transition in terms of its component structure around $p=\frac{1}{d}$, where $d$ is the degree of the product graph, and that this phase transition is quantitatively similar to that of $G(n,p)$. Given a constant $\epsilon>0$, let us define $y\coloneqq y(\epsilon)$ to be the unique solution in $(0,1)$ of the equation
\begin{align}\label{survival prob}
    y=1-\exp\left(-(1+\epsilon)y\right).
\end{align}
\begin{thm}[Theorem 2 in \cite{DEKK22}]\label{th: dekk22}
Let $C>1$ be a constant and let $\epsilon>0$ be sufficiently small. For all $j\in [t]$, let $G^{(j)}$ be a connected regular graph of degree $d_j$
such that $1<\big|V\left(G^{(j)}\right)\big|\le C$. Let $G=\square_{j=1}^{t}G^{(j)}$, let $n\coloneqq |V(G)|$ and let $p=\frac{1+\epsilon}{d}$, where $d\coloneqq d(G) =\sum_{j=1}^{t}d_j$ is the degree of $G$. Then, \textbf{whp}\footnote{With high probability, that is, with probability tending to $1$ as $t$ tends to infinity.},
there exists a unique component of order $\left(1+o(1)\right)yn$ in $G_p$, where $y=y(\epsilon)$ is defined as in (\ref{survival prob}). Furthermore, \textbf{whp}, all the remaining components of $G_p$ are of order $O_{\epsilon, C}(d)$.
\end{thm}

This can perhaps be viewed as an example of the \emph{universality} of the phase transition that $G(n,p)$ undergoes --- in many percolation models various aspects of the phase transition close to the critical point seem to behave in a quantitatively similar manner, under the right rescaling, independently of the host graph (see, for example \cite{HH17}). In this case, the proportion $y$ of the host graph $G$ which is covered by the giant component is the same as arises in $G(n,p)$ \cite{ER60}, but also in supercritical percolation in the hypercube \cite{AKS81,BKL92}, pseudo-random graphs \cite{FKM04}, and many other percolation models, see for example \cite{MR95,BJR07}.

Whilst the internal structure of the giant component in $G(n,p)$ is reasonably well understood, in other percolation models, such as hypercube percolation, many basic questions about the structure of the giant component remain unanswered, although in light of this universality phenomena there are natural conjectures suggested by the structure in $G(n,p)$. Since the expansion properties of the giant component in $G(n,p)$ have been key to understanding its likely structural properties, in order to better understand the structure of the giant component in percolated high-dimensional product graphs it is natural to ask about its expansion properties, and in order to answer this question it seems crucial to understand first the isoperimetric properties of general high-dimensional product graphs.

A well-known result of Chung and Tetali \cite{CT98} (see also Tillich \cite{T00}) shows that, at least on a broad scale, the isoperimetric properties of a product graph are closely related to those of the base graphs.
\begin{thm}[Theorem 2 of \cite{CT98}]\label{th: CT98}
Let $G^{(1)},\ldots, G^{(t)}$ be such that $|V(G^{(j)}|>1$ for all $j\in [t]$. Let $G = \square_{j=1}^{t}G^{(j)}$. Then
\[
\min_j \left\{i\left( G^{(j)}\right) \right\} \geq i(G) \geq \frac{1}{2}\min_j \left\{i\left( G^{(j)}\right) \right\}.
\]
\end{thm}
However, on a finer scale we might expect smaller sets in a product graph to expand by a larger factor than is suggested by Theorem \ref{th: CT98}. Indeed, in the case of the hypercube, Theorem \ref{th: CT98} gives a much weaker bound on the expansion of small sets than is implied by Theorem \ref{th: Harper}, where the expansion of small sets is asymptotically optimal, and this optimal expansion is critical to understanding the distribution of small percolation clusters in $Q^d_p$. It is thus natural to ask whether similar isoperimetric results hold on a finer scale for arbitrary product graphs.

\subsection{Main results}
Our first main results are two edge-isoperimetric inequalities for high-dimensional product graphs, under mild assumptions on the base graphs. The first concerns high-dimensional product graphs whose base graphs are bounded and regular.
\begin{theorem}\label{th: iso 1}
    Let $C>1$ be an integer. For all $j \in [t]$, let $G^{(j)}$ be a $d_j$-regular graph with $1<|V(G^{(j)})|\le C$. Let $G=\square_{j=1}^tG^{(j)}$, let $n\coloneqq|V(G)|$ and let $d \coloneqq \sum_{j=1}^t d_j$. Then for any $k\in[n]$,
   \begin{align*}
        i_k(G) \ge d-(C-1)\log_2 k.
    \end{align*}
\end{theorem}
Observe that if $\log_2 k\ll d$, then Theorem \ref{th: iso 1} implies that $i_k(G) \ge (1-o(1))d$ which, since $G$ is $d$-regular, is asymptotically optimal. Furthermore, in the particular case of $Q^d$ we have that $C=2$, and this result recovers the tight bound for the hypercube (Theorem \ref{th: Harper}). Note, however, that when the base graphs are larger, there are $k \in [n]$ with $(C-1)\log_2 k > d$, for which Theorem \ref{th: iso 1} gives a trivial bound for $i_k(G)$.

The second isoperimetric inequality holds for high-dimensional product graphs whose base graphs are bounded and connected (and not necessarily regular), and gives an effective bound for larger values of $k$.
\begin{theorem}\label{th: iso 2}
    Let $C>1$ be an integer. For all $j \in [t]$, let $G^{(j)}$ be a connected graph with $1< |V(G^{(j)})|\le C$. Let $G=\square_{j=1}^tG^{(j)}$ and let $n\coloneqq|V(G)|$. Then for any $k\in[n]$,
    \begin{align*}
        i_k(G) \ge \frac{1}{C-1} \log_C \left(\frac{n}{k}\right).
    \end{align*}
\end{theorem}
Note that, taking $C=2$, Theorem \ref{th: iso 2} also implies the classical edge-isoperimetric bound for the hypercube. Furthermore, in general Theorem \ref{th: iso 2} implies that $i_k(G)=\Omega\left(\ln \left(\frac{n}{k}\right)\right)$ for all $k\in[n]$, which recovers the asymptotic result of Tillich \cite{T00} on high-dimensional Cartesian \textit{powers} of graphs, which was proved using analytic methods inspired by isoperimetric problems in Riemannian geometry. Let us also mention a related result of Lev \cite{L15} which shows that $i_k(G)$ has the same asymptotic growth rate in any Abelian Cayley graph, where the implicit constant depends on the exponent of the underlying group. Moreover, for many different types of product graphs where the isoperimetric problem has been studied, among them Hamming graphs \cite{B99} and the $d$-dimensional torus graphs \cite{C02}, the bound given by Theorem \ref{th: iso 2} is known to be \textit{asymptotically tight} up to a multiplicative constant. In fact, as we will discuss in more detail in Section \ref{discussion}, it can be shown that Theorem \ref{th: iso 2} is asymptotically tight for \emph{any} high-dimensional product graph all of whose base graphs are isomorphic.

Using these new isoperimetric inequalities, we are able to derive several likely expansion properties of the giant component after percolation in a high-dimensional product graph whose base graphs are regular and of bounded order. These typical expansion properties which we will present, and their consequences, not only generalise but also improve substantially upon the known typical bounds in $Q^d_p$ given in \cite{EKK22}. We note that while we present the results in the supercritical regime, that is when $\epsilon>0$ is a small constant and $p=\frac{1+\epsilon}{d}$, the results naturally extend (with slight adaptations in the statements) to the sparse regime, that is, when $p=\frac{c}{d}$ for constant $c>1$.

Given a graph $G$, a subset $S\subseteq V(G)$ and $r\in \mathbb{N}$, we denote by $N^{r}_G(S)$ the $r$-\textit{th external neighbourhood} of $S$ in $G$, that is, the set of vertices in $V(G)\setminus S$ which are at distance at most $r$ from $S$ in $G$. When $r=1$, we omit the superscript. 

\begin{theorem}\label{th: expansion}
    Let $C>1$ be an integer. For all $j\in [t]$, let $G^{(j)}$ be a $d_j$-regular connected graph with $1<|V(G^{(j)})|\le C$. Let $G=\square_{j=1}^tG^{(j)}$, let $n\coloneqq|V(G)|$ and let $d \coloneqq \sum_{j=1}^t d_j$. Let $\epsilon>0$ be a small enough constant and let $p=\frac{1+\epsilon}{d}$. Let $L_1$ be the largest component in $G_p$. Then, there exists a positive constant $c=c(\epsilon)$ such that \textbf{whp},
    \begin{enumerate}[(a)]
        \item\label{i:all} for all $k\le \frac{3\epsilon n}{2}$ and all subsets $S\subseteq V(L_1)$ with $|S|=k$,
        \begin{align*}
             |\partial_{G_p}(S)|\ge \frac{c|S|}{d\ln d};
        \end{align*}
        \item\label{i:large} for all $\epsilon^2n\le k \le \frac{3\epsilon n}{2}$ and all subsets $S\subseteq V(L_1)$ with $|S|=k$,
        \begin{align*}
            |N_{G_p}(S)|\ge \frac{c|S|}{d\ln d}.
        \end{align*}
    \end{enumerate}
\end{theorem}
We note that, since $G$ is $d$-regular, Theorem \ref{th: expansion}\ref{i:all} implies a lower bound of $\Omega\left(\frac{1}{d^2\ln d}\right)$ on the vertex-expansion of arbitrary subsets of $L_1$. Theorem \ref{th: expansion}\ref{i:large} then improves this by a factor of $d$ for linear-sized sets.

If we make the additional assumption that our subset $S\subseteq V(L_1)$ is \emph{connected}, that is, $G_p[S]$ is connected, then we are able to give stronger bounds on the expansion.
\begin{theorem}\label{th: connected expansion}
    Let $C>1$ be an integer. For all $j\in [t]$, let $G^{(j)}$ be a $d_j$-regular connected graph with $1<|V(G^{(j)})|\le C$. Let $G=\square_{j=1}^tG^{(j)}$, let $n\coloneqq|V(G)|$ and let $d \coloneqq \sum_{j=1}^t d_j$. Let $\epsilon>0$ be a small enough constant and let $p=\frac{1+\epsilon}{d}$. Let $L_1$ be the largest component in $G_p$. Then, there exists a positive constant $c=c(\epsilon)$ such that \textbf{whp} for all subsets $S\subseteq V(L_1)$ with $|S|=k$ and $G_p[S]$ connected,
    \begin{enumerate}[(a)]
        \item\label{i:small} for all $\frac{9\ln c \cdot d}{\epsilon^2}\le k \le n^{\epsilon^5}$,
        \begin{align*}
            |N_{G_p}(S)|\ge c|S|;
        \end{align*}
        \item\label{i:medium} for all $n^{\epsilon^5}\le k\le \frac{3\epsilon n}{2}$,
        \begin{align*}
             |\partial_{G_p}(S)|\ge \frac{c|S|\ln\left(\frac{n}{|S|}\right)}{d\ln d}.
        \end{align*}
    \end{enumerate}
\end{theorem}
One interesting interpretation of Theorem \ref{th: connected expansion}, noting that the bound in Theorem \ref{th: connected expansion}\ref{i:small} implies the bound in Theorem \ref{th: connected expansion}\ref{i:medium} for the same range of $k$, is as a \emph{sparsification} of Theorem \ref{th: iso 2}, and so in the particular case of the hypercube a sparsification of Harper's theorem. In other words, recalling that we are interested in percolation with probability $p=\Theta\left(\frac{1}{d}\right)$, broadly Theorem \ref{th: connected expansion} tells us that, if we restrict ourselves to connected subsets which are not too small, then the naive isoperimetric inequality that holds in expectation in $G_p$ by Theorem \ref{th: iso 2} for a given set, actually holds \textbf{whp} up to a logarithmic factor for all sets simultaneously. We note that the restriction to large connected sets here is necessary, due to the likely existence of bare paths of length $\Theta(d)$ in $G_p$, which can be shown by elementary arguments, which are connected but exhibit poor expansion, and in fact the likely existence of a disjoint family of such paths of large total volume.

Let us make a few clarifying remarks about Theorems \ref{th: expansion} and \ref{th: connected expansion}. We note first that \textit{Theorem \ref{th: connected expansion} implies Theorem \ref{th: expansion}\ref{i:all}}. Indeed, given such a (not necessarily connected) set $S$, each component $K$ of $G_p[S]$ either has order at least $\frac{9d \ln C}{\epsilon^2}$, and hence by Theorem \ref{th: connected expansion} \textbf{whp} has edge-boundary at least $c\frac{|K|}{d\ln d}$, or has size at most $\frac{9d \ln C}{\epsilon^2}$ and at least one edge in its boundary, since $L_1$ is connected, and hence has edge-boundary of order $\Omega\left(\frac{|K|}{d}\right)$. Since the edge-boundaries for different components are disjoint, the claim follows.

We note further that the results in Theorems \ref{th: expansion} and \ref{th: connected expansion} are (almost-)optimal for a wide range of choices of $k$. Indeed, since $|N_G(S)| \leq |S|d$ for all subsets $S$, a simple first-moment calculation shows that Theorem \ref{th: connected expansion}\ref{i:small} is optimal up to the constant factor. Moreover, Theorem \ref{th: connected expansion}\ref{i:medium} (and hence also Theorem \ref{th: expansion}\ref{i:large}) are optimal up to the logarithmic factor in $d$. Indeed, consider the particular example of $Q^d_p$ and let $Q'$ be the subcube of $Q^d$ obtained by fixing the first $\log_2 x$ coordinates to be $0$, noting that $|V(Q')|=\frac{n}{x}\coloneqq k$, and that every vertex in $Q'$ is adjacent to at most $\log_2 x =\log_2\left(\frac{n}{k}\right)$ vertices in $Q^d\setminus Q'$. Therefore, by a Chernoff-type bound, \textbf{whp} the edge-boundary of $V(Q')$ (and hence its vertex-boundary) in $Q^d_p$ has order $O\Big(\frac{k \ln \left(\frac{n}{k}\right)}{d}\Big)$. In particular, if $\log_2x\ll \epsilon d$, then $Q'_p$ is supercritical and contains a connected  subset $S$ of order $\Theta(k)$, whose edge-boundary (and hence vertex-boundary) has size at most that of $V(Q')$, and hence is of order $O\Big(\frac{k \ln \left(\frac{n}{k}\right)}{d}\Big) = O\Big(\frac{ |S|\ln \left(\frac{n}{|S|}\right)}{d}\Big)$.

Finally, it is worth comparing Theorem \ref{th: expansion} with the expansion properties of the giant component of $Q^d_p$, as given in \cite{EKK22}. There, it was shown that for any set $S\subseteq V(L_1)$, \textbf{whp} $|N_{G_p}(S)|=\Omega\left(\frac{|S|}{d^5}\right)$, and for linear-sized subsets $S$ \textbf{whp} $|N_{G_p}(S)|=\Omega\left(\frac{|S|}{d^2\ln d}\right)$. In comparison, as mentioned above, it follows from Theorem \ref{th: expansion} that \textbf{whp} for any set $S\subseteq V(L_1)$, $|\partial_{G_p}(S)|=\Omega\left(\frac{|S|}{d\ln d}\right)$ and ${|N_{G_p}(S)|=\Omega\left(\frac{|S|}{d^2\ln d}\right)}$, and for linear-sized subsets, \textbf{whp} $|N_{G_p}(S)|=\Omega\left(\frac{|S|}{d\ln d}\right)$. 

A particularly interesting consequence that we can derive from Theorem \ref{th: expansion}\ref{i:large} is that typically $L_1$ contains a linear-sized subgraph which is a good expander at all scales.
\begin{theorem}\label{th: expander}
    Let $C>1$ be an integer. For all $j\in [t]$, let $G^{(j)}$ be a $d_j$-regular connected graph with $1<|V(G^{(j)})|\le C$. Let $G=\square_{j=1}^tG^{(j)}$, let $n\coloneqq|V(G)|$ and let $d \coloneqq \sum_{j=1}^t d_j$. Let $\epsilon>0$ be a small enough constant and let $p=\frac{1+\epsilon}{d}$. Let $L_1$ be the largest component in $G_p$. Then, there exists a positive constant $c=c(\epsilon)$ such that \textbf{whp} the following holds. There exists a subgraph $H\subseteq L_1$ such that $|V(H)|\ge \frac{3\epsilon n}{2}$, and for every $S\subseteq V(H)$ with $|S|\le \frac{|V(H)|}{2}$,
     \begin{align*}
         |N_{H}(S)|\ge \frac{c|S|}{d\ln d}.
     \end{align*}
\end{theorem}
\begin{remark}
    The fraction $\frac{3}{2}$ in Theorem \ref{th: expander} can be replaced by any constant strictly smaller than $2$. In particular, since \textbf{whp} $|V(L_1)| = \left(2\epsilon - O(\epsilon^2)\right)n$, we can choose an $H$ which covers almost all of the vertices of $L_1$.
\end{remark}
We note that, in the case of the hypercube, as shown in \cite[Claim 5.2]{EKK22}, \textbf{whp} every linear-sized subgraph of the giant component in a supercritical $Q^d_p$ has edge-expansion $O\left(\frac{1}{d}\right)$, and thus Theorem \ref{th: expander} is optimal up to the logarithmic factor in $d$. In the case of $G(n,p)$, Benjamini, Kozma and Wormald \cite{BKW14}, and Krivelevich \cite{K18} showed that in the supercritical regime there is typically a linear-sized subgraph $H$ of the giant component with a constant edge- and vertex-expansion (see also \cite{DLP14}). This result and the accompanying structural description of the giant component in terms of this expanding subgraph given by Benjamini, Kozma and Wormald \cite{BKW14}, can be used to determine the asymptotic order of many important structural parameters of the giant component in $G(n,p)$. An analogous description of the structure of the giant component in a percolated high-dimensional product graph is likely to be useful for determining its finer structure. 

Using Theorems \ref{th: connected expansion} and \ref{th: expansion}\ref{i:large}, we can obtain several interesting consequences on the typical structure of $L_1$. 
\begin{theorem}\label{th: consequences}
    Let $C>1$ be an integer. For all $j\in [t]$, let $G^{(j)}$ be a $d_j$-regular connected graph with $1<|V(G^{(j)})|\le C$. Let $G=\square_{j=1}^tG^{(j)}$, let $n\coloneqq|V(G)|$ and let $d \coloneqq \sum_{j=1}^t d_j$. Let $\epsilon>0$ be a small enough constant, let $p=\frac{1+\epsilon}{d}$, and let $L_1$ be the largest component of $G_p$. Then \textbf{whp}, 
    \begin{enumerate}[(a)]
        \item\label{i: diameter} the diameter of $L_1$ is $O(d\ln^2 d)$; 
        \item\label{i: mixing time} the mixing time of a lazy random walk on $L_1$ is $O(d^2\ln^2d)$; 
        \item\label{i: cycle} the circumference of $L_1$ is $\Omega\left(\frac{n}{d \ln d}\right)$.
    \end{enumerate}
\end{theorem}

The bounds given in Theorem \ref{th: consequences}\ref{i: diameter}, \ref{i: mixing time}, and \ref{i: cycle} are close to optimal, up to a multiplicative factor of $\ln^2d$ in the first two cases and of $d \ln d$ in the latter case. In the case of Theorem \ref{th: consequences}\ref{i: cycle} this is immediate, and in the other two cases, this follows from the likely existence in $G_p$ of a bare path of length $\Omega(d)$. 
Furthermore, note that these typical bounds not only generalise but also improve substantially upon the typical bounds in $Q^d_p$ given in \cite{EKK22}.

The structure of the paper is as follows. In Section \ref{outline}, we provide an outline of the proofs of our main results, stressing also the main challenges one needs to overcome, our approach towards them and the key novelties of this paper. In Section \ref{lemmas}, we present and establish several lemmas that will be useful for us throughout the paper. In Section \ref{iso}, we prove Theorems \ref{th: iso 1} and \ref{th: iso 2} (the reader who is interested in the implications of our results for the hypercube can recall Harper’s inequality: $i_{k}(Q^d)\ge d-\log_2k$, think of our base graphs as $K_2$, and skip Section \ref{iso}). In Section \ref{expansion} we prove Theorems \ref{th: expansion}, \ref{th: connected expansion} and \ref{th: expander}. In Section \ref{s: consequence} we prove Theorems \ref{th: consequences}\ref{i: cycle}, \ref{i: diameter} and \ref{i: mixing time}. Finally, in Section \ref{discussion} we mention some questions and open problems.

\section{Outline of the proofs}\label{outline}
For the proof of Theorem \ref{th: iso 1}, since we assume that the graph is regular, it suffices to bound from above the density of any set of a given size. To that end, we can use the product structure of $G$ to decompose it into \textit{disjoint projections of lower dimension}. Then, given a subset $S\subseteq V(G)$, this decomposition of $G$ induces a partition of $S$, and we can express the density of $S$ as a function of the density inside each partition class and the density between the partition classes. Since each partition class lives in a lower dimensional projection, we can bound its density inductively. However, whilst our desired bound is subadditive, we require a stronger inequality (Corollary \ref{c: iso 1}) to account for the cross-partition density, which we prove using a novel entropic argument (Lemma \ref{l: iso 1}).

The proof of Theorem \ref{th: iso 2} also utilises the entropy function. More explicitly, given a subset $S \subseteq V(G)$, we consider a uniformly chosen random vertex in $S$, which we can consider as a random vector in the product space $V(G)$. It can be shown that the entropy of suitable projections of this random vector can be bounded in terms of the edge boundary of $S$ in a fixed direction. We can then combine these individual bounds into a bound for $\partial(S)$ in terms of $|S|$ using Shearer's inequality.

Moving to our results on typical expansion properties of the giant component $L_1$, for small enough sets, we can combine our almost tight isoperimetric inequality (Theorem \ref{th: iso 1}) with good bounds on the number of connected subsets of $G$ (Lemma \ref{l: trees}) to argue via a first-moment calculation that it is unlikely that any small connected subset of $L_1$ does not expand well. This allows one to derive Theorem \ref{th: connected expansion}\ref{i:small}. In the proof of Theorem \ref{th: connected expansion}\ref{i:small} there is a trade-off, in (\ref{e:balance}), between the \emph{enumerative bound} of the number of connected sets of size $k$, and the \emph{probability bound} that these sets have small expansion, which is related to the isoperimetric inequality. For larger sets, the strategy of Theorem \ref{th: connected expansion}\ref{i:small} is ineffective because of the limitations of the isoperimetric inequality, leading to a weaker probability bound, and for disconnected sets the strategy is ineffective due to a weaker enumerative bound, as there are many more disconnected sets than connected sets. 

Thus, our key improvements come in the proof of Theorem \ref{th: connected expansion}\ref{i:large} and Theorem \ref{th: expansion}, and therein lie several novel techniques, embedded in two key lemmas: Lemma \ref{l: key lemma} and Lemma \ref{l: key lemma 2}. We argue via a \emph{two-round exposure}. Setting $\delta=\delta(\epsilon)\ll \epsilon$ we define $p_2=\frac{\delta}{d}$ and let $p_1$ be such that $(1-p_1)(1-p_2)=1-p$, so that $G_p$ has the same distribution as $G_{p_1}\cup G_{p_2}$, noting that $p_1\ge \frac{1+\epsilon-\delta}{d}$. By Theorem \ref{th: dekk22}, we know that \textbf{whp} $G_{p_1}$ already contains a giant component of linear order, which we denote by $L_1'$. Furthermore, note that \textbf{whp}, $L_1'$ will be a subgraph of $L_1$, the giant component of $G_p$ (in fact, typically it will cover most of the vertices of $L_1$). We thus informally refer to $L_1'$ as the \textit{early giant}.

Key ideas of \cite{EKK22}, which we generalise to the setting of high-dimensional product graphs, use an isoperimetric inequality (Theorem \ref{th: iso 2}) to give a strong probability bound for the event that a given subset of $L_1'$ does not expand well after sprinkling (see Lemma \ref{l: disjoint paths} and \cite[Lemma 3.4]{EKK22}). However, a naive enumerative bound on the number of subsets of $L'_1$ is too weak to conclude that \textbf{whp} \textit{every} subset of $L_1'$ expands well, using a union bound. 

An essential contribution here is then a novel double counting argument to improve this enumerative bound. Indeed, we only need to demonstrate an expansion property for the subsets of $L_1'$ which do not already expand inside $L_1'$, where the required expansion factor is $o_d(1)$. In particular, for each such set $S$ the size of its boundary $B$ in $L_1'$ is significantly smaller than the size of $S$, and so naively, enumerating over the set of possible boundaries should be more effective than enumerating over the sets themselves. Of course, there may be many sets $S$ with the same boundary, but we will see that again the assumption that $S$ does not expand well will allow us to give an effective bound on the number of relevant $S$ with boundary $B$ (see Lemma \ref{l: key lemma}).

Naturally, the subsets we consider can contain many vertices from the residue $L_1-L_1'$, and thus showing good expansion of subsets of the early giant $L_1'$ in $L_1$ does not immediately imply good expansion in $L_1$. Our second key contribution then lies in the analysis of the typical structure of subsets in the residue, and in particular their likely expansion into the early giant (see Lemmas \ref{l: Kd} and \ref{l: key lemma 2}). Having all these tools at hand, we prove Theorems \ref{th: expansion} and \ref{th: connected expansion}. 

The proof of Theorem \ref{th: expander} uses ideas from \cite{K19} to move from expansion at a fixed scale to expansion at all scales, together with our expansion result on large sets (Theorem \ref{th: expansion}\ref{i:large}). Having found a large expander subgraph, one can then derive the existence of a long cycle (Theorem \ref{th: consequences}\ref{i: cycle}) using techniques from \cite{K19a}. For Theorem \ref{th: consequences}\ref{i: diameter}, we analyse the growth rate of a ball of given radius. To obtain tight results, we use the edge-expansion of connected sets given in Theorem \ref{th: connected expansion}, together with the fact that, typically, connected subsets of the random subgraph $G_p$ are not dense, and thus edge-expansion is tightly connected to vertex expansion (see \cite{KRS15} for similar ideas). Finally, Theorem \ref{th: consequences}\ref{i: mixing time} follows from a careful analysis of the method of Fountoulakis and Reed together with our results on the expansion of connected sets (see \cite{FR08, KRS15, DK21B} for somewhat similar implementations).

\section{Preliminary lemmas}\label{lemmas}
We will use the following standard Chernoff type bound on the tail probabilities of the binomial distribution (see, for example, Appendix A in \cite{AS16}):
\begin{lemma}\label{l: Chernoff}
Let $N\in \mathbb{N}$, let $p\in [0,1]$, and let $X\sim Bin(N,p)$. Then for any $b>0$, 
\begin{align*}
    \mathbb{P}\left(X\ge bNp\right)\le \left(\frac{e}{b}\right)^{bNp}.
\end{align*}
\end{lemma}

We will also use the well-known Azuma-Hoeffding inequality (see, for example, Chapter 7 in \cite{AS16}),
\begin{lemma}\label{l: AH}
    Let $X = (X_1,X_2,\ldots, X_m)$ be a random vector with range $\Lambda = \prod_{i \in [m]} \Lambda_i$ and let $f:\Lambda\to\mathbb{R}$ be such that there exists $D \in \mathbb{R}_+$ such that for every $x,x' \in \Lambda$ which differ only in the $j$-th coordinate,
    \begin{align*}
        |f(x)-f(x')|\le D.
    \end{align*}
    Then, for every $b\ge 0$,
    \begin{align*}
        \mathbb{P}\left[\big|f(X)-\mathbb{E}\left[f(X)\right]\big|\ge b\right]\le 2\exp\left(-\frac{b^2}{2mD^2}\right).
    \end{align*}
\end{lemma}

We require the following bound on the number of $k$-vertex trees in a $d$-regular graph
$G$, which follows immediately from Lemma 2 in \cite{BFM98}.
\begin{lemma}\label{l: trees}
Let $k \in \mathbb{N}$ and let $t_k(G)$ be the number of trees on $k$ vertices which are subgraphs of an $n$-vertex $d$-regular graph $G$. Then
\begin{align*}
     t_k(G)\le n(ed)^{k-1}.
\end{align*}
\end{lemma}

In certain situations it will be useful to decompose a tree into connected parts of roughly equal size. In \cite[Lemma 2.2]{EKK22} and \cite[Proposition 4.5]{KN06}, such a result is given where the tree is decomposed into vertex-disjoint subsets, but where the gap between the sizes of the subsets grows with the maximum degree of the tree. For our purposes, we will require a similar result with tighter control over the size of the parts. To do so, we instead decompose into edge-disjoint subsets, which allows us to bound the difference in the sizes of the subsets independently of the tree.
\begin{lemma}\label{l: tree decomposition}
Let $\ell>0$ be an integer. Let $T$ be a tree with $|V(T)|\ge \ell$. Then, there exist vertex sets $A_1,\ldots A_s$ such that:
\begin{enumerate}[(a)]
    \item $V(T)=\bigcup_{1\le i \le s}A_{i}$; \label{i: partition}
    \item $T[A_i]$ is connected for all $i\in [s]$;  \label{i: connected}
    \item $|A_i\cap \left(\bigcup_{j\in([s]\setminus \{i\})}A_{j}\right)|\le 1$; and \label{i: edge-disjoint}
    \item $\ell\le |A_{i}|\le 3\ell$ for all $i\in [s]$. \label{i: sizes}
\end{enumerate}
\end{lemma}
\begin{proof}
We prove the result by induction on $m=|V(T)|$. If $\ell\le m\le 3\ell$, the trivial partition $A_1=V(T)$ satisfies the conclusion of the lemma. Suppose then that $m>3\ell$, and that the statement holds for all trees $T'$ where $\ell \le |V(T')|<m$. Let us choose an arbitrary root $w\in V(T)$ for $T$. For each $v \in V(T)$, we write $T_v$ for the subtree of $T$ rooted at $v$. 

Let $v$ be a vertex of maximal distance from $w$ such that $|V(T_{v})|\ge\ell$. Note that by our choice of $v$, $|V(T_x)|<\ell$ for every child $x$ of $v$. Then, there exists a subset of the children of $v$, $X_0\subseteq V(T)$, such that $\ell-1\le \sum_{x\in X_0}|V(T_x)|\le 2\ell -2$. Set $A_1=\{v\}\cup\bigcup_{x\in X_0}V(T_{x})$, and note that $\ell \le |A_1| \le 2\ell -1$ and that $T[A_1]$ is connected. Set $T'=T\setminus\bigcup_{x\in X_0}V(T_x)$, and note that $T'$ is connected with $|T|>|T'|\ge |T|-(2\ell-2)\ge \ell$. We may thus apply the induction hypothesis to $T'$, producing $A_2,\ldots, A_s$ satisfying properties \ref{i: partition} through \ref{i: sizes} with respect to $T'$.

Consider the sets $A_1, A_2, \ldots, A_s$ with respect to $T$. Properties \ref{i: partition}, \ref{i: connected} and \ref{i: sizes} are clear from the above construction. Since $V(T) \cap V(T') = \{v\}$, $v$ is the only vertex that can be shared by $A_1$ and any $A_j$ with $j>1$, and so property \ref{i: edge-disjoint} is satisfied as well.
\end{proof}

The following theorem will allow us to deduce the existence of a long \textit{cycle} in a graph with good vertex-expansion.
\begin{thm}\cite[Theorem 1]{K19a}\label{th: krivelevich cycle}
Let $a\ge 1, b\ge 2$ be integers. Let $G$ be a graph on more than $a$ vertices satisfying
\begin{align*}
    |N(S)|\ge b, \quad \text{for every } S\subseteq V(G) \text{ with } \frac{a}{2}\le|S|\le a.
\end{align*}
Then $G$ contains a cycle of length at least $b+1$.
\end{thm}

Given a discrete random variable $X$ taking values in some range $\mathcal{X}$, the entropy of $X$ is given by
\begin{align*}
    H(X)\coloneqq \sum_{x\in\mathcal{X}}-p(x)\log_2p(x),
\end{align*}
where $p(x)\coloneqq\mathbb{P}(X=x)$ and we follow the convention that $x\log_2x=0$ for $x=0$. Given discrete random variables $X_1,X_2,\ldots, X_t$, the joint entropy $H(X_1,X_2,\ldots, X_t)$ is defined to be the entropy of the random vector $(X_1,X_2,\ldots, X_t)$. We denote by 
\[
H\left(X_1,X_2,\ldots, X_t|X_{t+1}\right)\coloneqq H\left(X_1,X_2,\ldots, X_{t+1}\right)-H\left(X_{t+1}\right)
\]
the conditional entropy of $(X_1,\ldots, X_t)$ given $X_{t+1}$. 
\begin{remark}\label{r: entropy}
    Observe that if $X_1$ determines $X_2$, then by definition
    \begin{equation}\label{eq:X1detX2}
    H(X_1|X_2)=H(X_1,X_2)-H(X_2)=H(X_1)-H(X_2).
 \end{equation} 
 Furthermore, if $X_3$ determines $X_2$, then \begin{equation}\label{eq:X3detX2}H(X_1,X_2|X_3)=H(X_1,X_2,X_3)-H(X_3)=H(X_1,X_3)-H(X_3)=H(X_1|X_3).
  \end{equation}
\end{remark}
Finally, given a random vector $X=(X_1,\ldots, X_t)$ and $I\subseteq[t]$, we denote by $X_I$ the random vector $(X_i)_{i\in I}$. 

We will require the following property of the entropy function due to Shearer (see, for example, \cite[Chapter 7]{AS16}):
\begin{lemma}[Shearer's inequality]\label{l: entropy}
Let $X_1, \ldots, X_t$ be discrete random variables and let $\mathcal{A}$ be a collection of (not necessarily distinct) subsets of $[t]$, such that each $i\in [t]$ is in at least $m$ members of $\mathcal{A}$. Then
\begin{align*}
    H(X_1,\ldots, X_t)\le\frac{1}{m}\sum_{A\in \mathcal{A}}H(X_A).
\end{align*}
\end{lemma}

Throughout the rest of the paper, unless explicitly mentioned otherwise, we assume that $C>1$ and ${G=\square_{j=1}^tG^{(j)}}$ is a high-dimensional product graph, whose base graphs $G^{(j)}$ are connected and $d_j$-regular with $1<|V(G^{(j)}|\le C$. Without loss of generality we can assume that $C\coloneqq C\left(G\right)=\max_{j\in[t]}\{|V(G^{(j)})|\}.$

We follow the notation regarding product graphs as in \cite{DEKK23}. Given a product graph $G=\square_{j=1}^tG^{(j)}$, we call the $G^{(j)}$ the \textit{base graphs} of $G$. Given a vertex $u = (u_1,u_2, \ldots, u_t)$ in $V(G)$ and $j \in [t]$ we call the vertex $u_j\in V(G^{(j)})$ the \textit{$j$-th coordinate} of $G$. Whenever confusion may arise, we will clarify whether the subscript stands for the enumeration of the vertices of the set, or for their coordinates.
When $G^{(j)}$ is a graph on a single vertex, that is, $G^{(j)}=\left(\{u\},\varnothing\right)$, we call it \textit{trivial} (and \textit{non-trivial}, otherwise). We define the \textit{dimension} of $G=\square_{j=1}^tG^{(j)}$ to be the number of base graphs $G^{(j)}$ of $G$ which are non-trivial (we note that the dimension of $G$ is not an invariant of $G$, and in fact depends on the choice of the base graphs). We note that $G$ is also regular, and we write $d\coloneqq \sum_{j=1}^t d_j$, which can be seen to be the degree of $G$, and let $n\coloneqq|V(G)|$. Furthermore, we assume in what follows that $\epsilon>0$ is a small enough constant, and let $p=\frac{1+\epsilon}{d}$. We denote by $G_p$ the graph obtained by retaining every edge of $G$ independently with probability $p$. 

Given a subgraph $H\subseteq G$, we denote by $d(H)$ the average degree of the subgraph $H$. Given two subsets $A, B\subseteq V(G)$ with $A\cap B=\varnothing$, we denote by $e(A,B)$ the number of edges between $A$ and $B$. Furthermore, given a subset $A\subseteq G$, we let $e(A)\coloneqq |E(G[A])|$. Finally, given a vertex $v\in V(G)$ and a subset $A\subseteq V(G)$, we denote by $d_A(v)$ the number of neighbours of $v$ in $A$.

We close this section with two lemmas about the structure of percolated product graphs. The first one is about large matchings in a random edge-subset, and is a fairly straightforward generalisation of Lemma 2.9 in \cite{EKK22}.
\begin{lemma}\label{l: matching}
Let $G$ be a $d$-regular graph. Let $c_1>0$ and $0<\delta<1$ be constants. Let $s\ge c_1d$. Let $F\subseteq E(G)$ be such that $|F|\ge s$ and let $q=\frac{\delta}{d}$. Then, there exists a constant $c_2=c_2(c_1,\delta)$ such that $F_{q}$, a random subset of $F$ obtained by retaining each edge independently with probability $q$, contains a matching of size at least $\frac{c_2s}{d}$ with probability at least $1-\exp\left(-\frac{c_2s}{d}\right)$.
\end{lemma}
\begin{proof}
We may assume $|F|=s$. If the matching number of $F_q$ is less than $\frac{c_2s}{d}$, then $F_q$ contains a maximal (by inclusion) matching of size $\ell< \frac{c_2s}{d}$. Let us then consider the number of maximal matchings in $F_q$ of size $\ell<\frac{c_2s}{d}$.
    
There are at most $\binom{|F|}{\ell}=\binom{s}{\ell}$ maximal matchings of size $\ell$ in $F$. Given a fixed matching $M$ of size $\ell$ in $F$, in order for it to be a maximal matching in $F_q$ its edges have to be retained, which happens with probability $q^\ell$, and there are no other edges in $F_q$ which are disjoint from $M$. Since $G$ is $d$-regular, there are at most $2\ell d$ edges which share a vertex with edges in $M$. Hence, there is a set of at least $|F|-2\ell d$ edges which do not appear in $F_q$, which happens with probability at most 
\begin{align*}
    (1-q)^{|F|-2\ell d}&\le \exp\left(-\frac{\delta s(1-2c_2)}{d}\right).
\end{align*}
Therefore, by the union bound, the probability that $F_q$ contains a maximal matching of size $\ell<\frac{c_2s}{d}$ is at most
\begin{align*}
    \sum_{\ell=0}^{\frac{c_2s}{d}}\binom{s}{\ell}\left(\frac{\delta}{d}\right)^\ell\exp\left(-\frac{\delta s(1-2c_2)}{d}\right)&\le \exp\left(-\frac{\delta s(1-2c_2)}{d}\right)\left(1+\sum_{\ell=1}^{\frac{c_2s}{d}}\left(\frac{e\delta s}{d\ell}\right)^{\ell}\right).
\end{align*}
Since $s\ge c_1d$ and for $c_2=c_2(c_1,\delta)$ small enough in terms of $c_1$ and $\delta$, the ratio of consecutive terms $\left(\frac{e\delta s}{d\ell}\right)^{\ell}$ is at least $2$, and hence the sum is dominated by the final term. Therefore, 
\begin{align*}
    \exp\left(-\frac{\delta s(1-2c_2)}{d}\right)\left(1+\sum_{\ell=1}^{\frac{c_2s}{d}}\left(\frac{e\delta s}{d\ell}\right)^{\ell}\right)&\le 3\exp\left(-\frac{\delta s(1-2c_2)}{d}\right)\left(\frac{e\delta}{c_2}\right)^{\frac{c_2s}{d}}\\
    &\le \exp\left(-\frac{c_2s}{d}\right),
\end{align*}
for small enough $c_2$.
\end{proof}

The second result bounds the typical number of high-degree vertices in $G_p$.
\begin{lemma}\label{l: high degree}
\textbf{Whp}, there are at most $\frac{n}{d^4}$ vertices of degree at least $\ln d$ in $G_p$.
\end{lemma}
\begin{proof}
Fix a vertex $v\in V(G)$. The degree of $v$ in $G_p$ is distributed according to $Bin(d,p)$. Thus, by Lemma \ref{l: Chernoff},
\begin{align*}
    \mathbb{P}\left(d_{G_p}(v)\ge \ln d\right)\le \left(\frac{e(1+\epsilon)}{\ln d}\right)^{\ln d}\le d^{-\frac{\ln\ln d}{2}}.
\end{align*}
Hence, the expected number of vertices in $G_p$ with degree at least $\ln d$ is at most $nd^{-\frac{\ln\ln d}{2}}$. Therefore, by Markov's inequality, \textbf{whp} there are at most $\frac{n}{d^4}$ vertices of degree at least $\ln d$ in $G_p$.
\end{proof}

\section{Isoperimetric inequalities}\label{iso}

The proofs of Theorems \ref{th: iso 1} and \ref{th: iso 2} will both use discrete entropy as a tool, but in quite different ways. For the proof of Theorem \ref{th: iso 1}, we require the following lemma bounding the entropy of a random variable from below.
\begin{lemma}\label{l: iso 1}
Let $C\ge 2$ be an integer and let $X$ be a random variable supported on $[C]$. For each $i\in [C]$, let $p(i)\coloneqq \mathbb{P}(X=i)$. Assume without loss of generality that $p(1) \leq p(2) \leq \ldots \leq p(C)$. Then
\[
\frac{C}{C-1}\left(1 - p(C)\right)\le H(X).
\]
\end{lemma}
\begin{proof}
We prove the result by induction on $C$. For $C=2$ we note that $0 \leq p(1) \leq p(2)$ and $p(1) + p(2) =1$, and so in particular $p(1)p(2) \leq \frac{1}{4}$. It follows that
\begin{align*}
H(X) &= p(1) \log_2  \frac{1}{p(1)} + p(2)\log_2 \frac{1}{p(2)} \geq p(1) \left(\log_2  \frac{1}{p(1)} + \log_2 \frac{1}{p(2)}\right) \\
&=  p(1) \left(\log_2 \frac{1}{p(1)p(2)}\right) \geq p(1) \log_2 4 \geq 2p(1) = 2\left(1-p(2)\right).
\end{align*}

Suppose that $C >2$. Let $Y$ be the indicator random variable of the event that $X=C$. Note that because $X$ determines $Y$, by Remark \ref{r: entropy},
\begin{align*}
    H(X)=H(X,Y)=H(Y)+H(X|Y).
\end{align*}
Let $q(1)\coloneqq \mathbb{P}(Y=1) = p(C)$ and $q(0)\coloneqq \mathbb{P}(Y=0) = \sum_{i=1}^{C-1} p(i)$.

If $q(1) \geq q(0)$, then by the induction hypothesis applied to $Y$ we can conclude that
\[
H(X) \geq H(Y) \geq \frac{C}{C-1}\left(1 - q(1)\right) = \frac{C}{C-1}\left(1-p(C)\right),
\]
as claimed.

Otherwise, again by the induction hypothesis applied to $Y$, we have that
\begin{align*}
    H(Y)\ge \frac{C}{C-1}\left(1-q(0)\right).
\end{align*}
Thus, we obtain that
\begin{align*}
H(X) &= H(Y) + H(X|Y) \\
&\geq \frac{C}{C-1}(1-q(0)) + \mathbb{P}(Y=1)H(X|Y=1) + \mathbb{P}(Y=0)H(X|Y=0).
\end{align*}
However, on the event $\{Y=1\}$ we have $X=C$, and so the second term is $0$, and by the induction hypothesis applied to the random variable $X$ conditional on $\{Y=0\}$, which is supported on $[C-1]$, we can conclude that
\[
H(X|Y=0) \geq \frac{C-1}{C-2}\left(1 - \frac{p(C-1)}{q(0)}\right).
\]

It follows that
\begin{align*}
H(X) &\geq \frac{C}{C-1}\left(1-q(0)\right) + \frac{C-1}{C-2}q(0)\left(1 - \frac{p(C-1)}{q(0)}\right)\\
&=\frac{C}{C-1}-\frac{C}{C-1}q(0) +\frac{C-1}{C-2}q(0)-\frac{C-1}{C-2}p(C-1)\\
&\ge \frac{C}{C-1}\left(1-p(C)\right).
\end{align*}
\end{proof}

An immediate corollary of Lemma \ref{l: iso 1} is the following inequality which is key to the proof of Theorem \ref{th: iso 1}.
\begin{corollary}\label{c: iso 1}
Let $C\ge 2$ be an integer and let $0\le k_1\le \cdots \le k_C$ and $k = \sum_{i=1}^C k_i$. Then
\begin{align*}
    \frac{C}{C-1}\left(k - k_C\right) + \sum_{i=1}^C k_i \log_2 k_i\le k \log_2 k.
\end{align*}
\end{corollary}
\begin{proof}
Let $X$ be a random variable supported on $[C]$ with $p(i) =\mathbb{P}(X=i)= \frac{k_i}{k}$ for each $i\in[C]$. Then, by the previous lemma,
\begin{align*}
\frac{C}{C-1}\left(1 - \frac{k_C}{k}\right) \leq  H(X) &= \sum_{i=1}^C \frac{k_i}{k} \log_2 \frac{k}{k_i} \\
&= \sum_{i=1}^C \frac{k_i}{k} \log_2 k - \sum_{i=1}^C \frac{k_i}{k} \log_2 k_i\\
&= \log_2 k - \sum_{i=1}^C \frac{k_i}{k} \log_2 k_i,
\end{align*}
which rearranges to give the claimed inequality.
\end{proof}

As will be seen in the proof of Theorem \ref{th: iso 1}, the inequality proven in Corollary \ref{c: iso 1} allows us to inductively bound the density of certain sets by considering an appropriate collection of projections. Using the regularity of the graph we can relate this density bound to an isoperimetric inequality.

\begin{proof}[Proof of Theorem \ref{th: iso 1}]
Let $k\coloneqq|S|$, and we may assume that $k\ge 2$. We claim that 
\begin{equation}\label{e:isoclaim}
\sum_{v\in S}d_{G[S]}(v)\le (C-1)k\log_2k. 
\end{equation}
Then, assuming that \eqref{e:isoclaim} holds, since $G$ is $d$-regular we obtain that 
\begin{align*}
    |\partial S|=|S|\left(d-d(G[S])\right)\ge k(d-(C-1)\log_2k),
\end{align*}
as required.

We prove \eqref{e:isoclaim} by induction on the dimension $t$ of the product graph $G$. For $t=1$, since $2\le k \le C$, we indeed have that 
\begin{align*}
    \sum_{v\in S}d_{G[S]}(v)\le \frac{k(k-1)}{2}\le (C-1)k\log_2k.
\end{align*}

Assume that \eqref{e:isoclaim} holds for all graphs of dimension $t'<t$. We may assume that, without loss of generality, $V(G^{(1)})=\{v_1,\ldots,v_C\}$. Let $H_1,\ldots, H_C$ be pairwise disjoint projections of $G$, such that $H_i$ is obtained by fixing the first coordinate of $G$ to be $v_i\in V(G^{(1)})$. Let $S_i=S\cap V(H_i)$ and set $k_i\coloneqq|S_i|$. Note that we have $\sum_{i=1}^Ck_i=k$, and we may assume without loss of generality that $k_1\le k_2 \le \ldots \le k_C$. Since each $H_i$ has dimension $t-1$, by the induction hypothesis, for all $1\le i \le C$,
\begin{align*}
    \sum_{v\in S_i}d_{G[S_i]}(v)=\sum_{v\in S_i}d_{H_i[S_i]}(v)\le (C-1)k_i\log_2k_i.
\end{align*}

Furthermore, observe that each vertex in $H_i$ has at most one neighbour in each $H_j$ for $j\neq i$. In particular, since $k_1\le k_2\le \ldots \le k_C$, it follows that $e(S_i, S_j)\le k_i$ whenever $i\le j$. Thus,
\begin{align*}
    \sum_{v\in S}d_{G[S]}(v)&=\sum_{i=1}^C\left(\sum_{v\in S_i}d_{G[S_i]}(v)+\sum_{j\neq i}e(S_i,S_j)\right)\\
    &\le \sum_{i=1}^C\left((C-1)k_i\log_2k_i+(C-i)k_i+\sum_{j< i}k_j\right)\\
    &\le \sum_{i=1}^C\Bigg((C-1)k_i\log_2k_i+(C-i)k_i+(i-1)k_{i-1}\Bigg)\\ 
    &\le C\left(k-k_C\right)+(C-1)\sum_{i=1}^Ck_i\log_2k_i.
\end{align*}
Therefore, we have by the above and by Corollary \ref{c: iso 1} that
\begin{align*}
    \sum_{v\in S}d_G[S](v)&\le (C-1)\left(\frac{C}{C-1}\left(k-k_C\right)+\sum_{i=1}^Ck_i\log_2k_i\right)\le (C-1)k\log_2k,
\end{align*}
as claimed.
\end{proof}

The proof of Theorem \ref{th: iso 2} will also utilise the entropy function, specifically Shearer's Lemma (Lemma \ref{l: entropy}) in a key way.
\begin{proof}[Proof of Theorem \ref{th: iso 2}]

Given $S\subseteq V(G)$, let $X$ be a uniformly distributed random variable on $S$, so that ${H(X)=\log_2|S|}$. Observe that we may consider $X$ as a random vector $X=(X_1,\ldots, X_t)$, where the random variables $X_i$ are given by the coordinates of the vertex $X\in V(G^{(1)})\times\cdots\times V(G^{(t)})$. For each $i\in [t]$ let $A_{-i}\coloneqq [t]\setminus \{i\}$ and let us set
\[
X_{-i}\coloneqq X_{A_{-i}} =  (X_1,\ldots, X_{i-1},X_{i+1},\ldots, X_t).
\]
Note that each $i\in [t]$ appears in exactly $t-1$ members of the family $\mathcal{A}=\left\{A_{-i}\colon i\in [t]\right\}$.

Thus, by Lemma \ref{l: entropy},
\begin{align}
    H(X)\le \frac{1}{t-1}\sum_{i=1}^tH(X_{-i}). \label{e:shearer}
\end{align}
Therefore, observing that $X$ determines $X_{-i}$ and $X_i$, we have by the above and by Remark \ref{r: entropy} that
\begin{align}\label{e:reduction}
    H(X)\stackrel{(\ref{e:shearer})}{\ge} \sum_{i=1}^t\left(H(X)-H(X_{-i})\right)\stackrel{(\ref{eq:X1detX2})}{=}\sum_{i=1}^tH(X|X_{-i})= \sum_{i=1}^tH(X_i,X_{-i}|X_{-i})\stackrel{(\ref{eq:X3detX2})}{=}\sum_{i=1}^tH(X_i|X_{-i}).
\end{align}

By definition,
\begin{align}\label{eq: entropy}
    H(X_i|X_{-i})=\sum_{x_{-i}}\mathbb{P}(X_{-i}=x_{-i})H(X_i|X_{-i}=x_{-i}) \eqqcolon \sum_{x_{-i}} w(x_{-i}),
\end{align}
where the sum ranges over the vectors $x_{-i}$ in the range of $X_{-i}$. 

Given such a point $x_{-i}$, there are $1\le r(x_{-i})\le C_i\coloneqq |V(G^{(i)})|$ vertices in $S$ whose projection is $x_{-i}$, where $\mathbb{P}(X_{-i}=x_{-i})=\frac{r(x_{-i})}{|S|}$. Then, since $X$ is uniformly distributed on $S$,
\begin{align*}
    H(X_i|X_{-i}=x_{-i})=\log_2r(x_{-i}).
\end{align*}
It follows that for each $x_{-i}$,
\[
w(x_{-i}) = \frac{r(x_{-i})\log_2r(x_{-i})}{|S|} \leq \frac{r(x_{-i})\log_2C_i}{|S|} \eqqcolon w'(x_{-i}),
\]
with equality if and only if $r(x_{-i})=C_i$.

However, since each $G^{(i)}$ is connected, for each $x_{-i}$ in the range of $X_{-i}$ where $r(x_{-i})<C_i$ there is at least one edge in the edge-boundary of $S$ in direction $i$. In particular, there are at most $|\partial_i(S)|$ many vectors $x_{-i}$ such that $r(x_{-i})<C_i$, where $\partial_i(S)$ denotes the edges in the edge-boundary of $S$ in the $i$-th direction, that is, that are obtained by changing the $i$-th coordinate of some $v\in S$. Furthermore, for each $x_{-i}$ with $r(x_{-i})<C_i$,
\[
w'(x_{-i}) - w(x_{-i}) \leq w'(x_{-i}) \le \frac{(C-1)\log_2C}{|S|}.
\]

Thus, by \eqref{eq: entropy}
\begin{align}
    H(X_i|X_{-i})&=\sum_{x_{-i}} w(x_{-i}) \nonumber\\
    &= \sum_{x_{-i}} w'(x_{-i}) +  \sum_{\substack{x_{-i} \\ r(x_{-i}) < C_i}} \big( w(x_{-i}) - w'(x_{-i})\big) \nonumber\\
    &\ge \log_2C_i-|\partial_i(S)|\frac{(C-1)\log_2C}{|S|}. \label{e:conditional}
\end{align}
Therefore, by \eqref{e:reduction} and \eqref{e:conditional},
\begin{align*}
    \log_2|S|=H(X)&\ge \sum_{i=1}^tH(X_i|X_{-i})\\
    &\ge \sum_{i=1}^t\left(\log_2C_i-|\partial_i(S)|\frac{(C-1)\log_2C}{|S|}\right)\\
    &\ge \log_2|V(G)|-|\partial(S)|\frac{(C-1)\log_2C}{|S|}.
\end{align*}
Rearranging, we obtain
\begin{align*}
    \frac{|\partial(S)|}{|S|}\ge \frac{\log_2|V(G)|-\log_2|S|}{(C-1)\log_2C}=\frac{1}{C-1}\log_C\left(\frac{|V(G)|}{|S|}\right),
\end{align*}
as claimed.
\end{proof}

\section{Expansion and Expanders}\label{expansion} 
We begin with the proof of the first part of Theorem \ref{th: connected expansion}. We note that the proof includes several elements similar to the proof of Lemma \ref{l: matching}.
\begin{proof}[Proof of Theorem \ref{th: connected expansion}\ref{i:small}]
We will assume that $c \leq \epsilon^4$. Given $\frac{7Cd}{\epsilon^2}\le k \le n^{\epsilon^5}$, let $\mathcal{A}_k$ be the event there exists a set $S \subseteq V(L_1)$ of order $k$ such that $S$ is connected in $G_p$ and $|N_{G_p}(S)| < c |S|$. Since $S$ is connected in $G_p$ it contains a spanning tree. Therefore, if $\mathcal{A}_k$ occurs, then there is some tree $T$ whose vertex set is $S$, all of whose edges are in $G_p$. By Lemma \ref{l: trees}, there are at most $n(ed)^{k-1}$ ways to choose the tree $T$, and the edges of $T$ are present in $G_p$ with probability $p^{k-1}$. 

Now, consider the auxiliary random bipartite graph $\Gamma(S,p)$, whose one side is $S$, the other side is $N_G(S)$, and we retain every edge of $G$ between $S$ and $N_G(S)$ in $\Gamma(S,p)$ independently with probability $p$. We then have that $|N_{G_p}(S)|\ge \nu \left(\Gamma(S,p)\right)$, where $\nu(H)$ is the matching number of $H$. Thus, it suffices to bound the probability that a maximum matching in $\Gamma(S,p)$ is smaller than $\epsilon^4k$, that is, 
\begin{align}
    \mathbb{P}\left(\mathcal{A}_k\right)&\le \sum_{\substack{S \subseteq V(G), |S|=k\\ T \text{ a tree} ,V(T)=S}}\mathbb{P}\left(\left(E(T)\subseteq E(G_p)\right)\land \left(\nu\left(\Gamma(S,p)\right)\le \epsilon^4k\right)\right) \nonumber\\
    &\le n(edp)^{k-1}\mathbb{P}\left(\nu\left(\Gamma(S,p)\right)\le \epsilon^4k\right). \label{e:balance}
\end{align}

Let us first bound the probability that $\nu\left(\Gamma(S,p)\right)=i$. This is, at most, the probability that $\Gamma(S,p)$ has an inclusion-maximal matching of size $i$. We have at most $\binom{kd}{i}$ ways to choose a matching $M$ of size $i$, and we then need to include the edges of the matching, which occurs with probability $p^i$. Due to the maximality of $M$, every edge of $G$ between $S$ and $N_G(S)$ disjoint from $M$ is not in $\Gamma(S,p)$. Thus, we have at least $|\partial(S)|-2id$ edges that do not fall into $\Gamma(S,p)$. Since $n \leq C^d$, by Theorem \ref{th: iso 1}
\begin{align*}
   |\partial(S)| \ge k\left(d-(C-1)\log_2k\right)&\ge k\left(d-(C-1)\cdot\log_2C\cdot\epsilon^5d\right)\ge (1-\epsilon^4)kd.
\end{align*}
Hence, by the union bound,
\begin{align*}
    \mathbb{P}\left(\nu\left(\Gamma(S,p)\right)=i\right)\le \binom{kd}{i}p^i(1-p)^{(1-\epsilon^4)kd-2id}.
\end{align*}
All in all, we obtain that
\begin{align*}
    \mathbb{P}\left(\mathcal{A}_k\right)&\le n(ed)^{k-1}p^{k-1}\sum_{i=0}^{\epsilon^4k}\binom{kd}{i}p^i(1-p)^{(1-\epsilon^4)kd-2id}\\
    &=n(edp)^{k-1}(1-p)^{(1-\epsilon^4)kd}\sum_{i=0}^{\epsilon^4k}\binom{kd}{i}p^i(1-p)^{-2id}\\
    &\le n\left((1+\epsilon)\exp\left(1-(1+\epsilon)(1-\epsilon^4)\right)\right)^k\left(1+\sum_{i=1}^{\epsilon^4k}\left(\frac{k(1+\epsilon)e}{i}\right)^i\exp\left(2(1+\epsilon)i\right)\right)\\
    &\le n\left((1+\epsilon)\exp\left(-\epsilon+2\epsilon^4\right)\right)^k\left(1+\sum_{i=1}^{\epsilon^4k}\left(\frac{e^4k}{i}\right)^i\right).
\end{align*}
Observe that the ratio of consecutive terms of $\left(\frac{e^4k}{i}\right)^i$ is at least $2$, and hence the sum is dominated by the last term. That is,
\begin{align*}
    \mathbb{P}\left(\mathcal{A}_k\right)&\le 2n\left((1+\epsilon)\exp\left(-\epsilon+2\epsilon^4\right)\right)^k\left(\frac{e^4}{\epsilon^4}\right)^{\epsilon^4k}\\
    &\le 2n\left((1+\epsilon)\exp\left(-\epsilon+\epsilon^3\right)\right)^k.
\end{align*}
Using $1+x\le \exp\left(x-\frac{x^2}{3}\right)$ for small enough $x>0$, together with $\ln n \le \ln C \cdot d$ (since $n\le C^t \le C^d$) and our assumption that $k\ge \frac{9\ln c \cdot d}{\epsilon^2}$, we obtain that
\begin{align*}
    \mathbb{P}\left(\mathcal{A}_k\right)\le  3n\exp\left(-\frac{\epsilon^2k}{4}\right)=o(1/n).
\end{align*}
Taking a union bound over the less than $n$ different values of $k$ completes the proof. 
\end{proof}

Throughout the rest of the section, we assume that $\epsilon>0$ is a small enough constant and let $\delta=\delta(\epsilon)\le \epsilon^3$ be a positive constant. We define $p_2=\frac{\delta}{d}$ and let $p_1$ be such that $(1-p_1)(1-p_2)=1-p$. We form $G_{p_i}$, $i\in \{1,2\}$, by including every edge of $G$ independently and with probability $p_i$. We set $G_1=G_{p_1}$ and $G_2=G_{p_2}\cup G_1$, so that $G_2$ has the same distribution as $G_p$. We note that by Theorem \ref{th: dekk22}, \textbf{whp} $G_{p_1}$ has a unique giant component, which we denote by $L_1'$, and that \textbf{whp} $G_{p}$ has a unique giant component which we denote by $L_1$, where $L_1'\subseteq L_1$.

\subsection{Expansion of subsets of the early giant}
We begin by showing likely expansion properties of subsets of the early giant. We will require the following density result.
\begin{lemma}[Lemma 4.7 in \cite{DEKK22}, rephrased]\label{l: density}
There exists a constant $c=c(\epsilon)>0$ such that \textbf{whp} every $v\in V(G)$ is at distance (in $G$) at most two from at least $cd^2$ vertices in $L_1'$.
\end{lemma}

The following lemma, which uses Lemma \ref{l: density} together with an edge-isoperimetric inequality for $G$ (Theorem \ref{th: iso 2}) and a result on
large matchings in a random edge-subset of G (Lemma \ref{l: matching}), gives a good bound on the probability that subsets of the early giant expand well after sprinkling.
\begin{lemma}\label{l: disjoint paths}
There exists a constant $c=c(\delta)>0$ such that the following holds. Let $A\cup B=V(L_1')$ be a partition of $V(L_1')$ with $\min\left\{|A|,|B|\right\}=k$. Then, with probability at least 
\begin{align*}
    1-\exp\left(-\frac{ck\ln\left(\frac{n}{k}\right)}{d}\right),
\end{align*}
there exists a family of $\frac{ck\ln\left(\frac{n}{k}\right)}{d}$ vertex-disjoint $A-B$ paths of length at most five in $G_{p_2}$.
\end{lemma}
\begin{proof}
By Lemma \ref{l: density}, there exists a constant $c'>0$ such that \textbf{whp} every $v\in V(G)$ is at distance (in $G$) at most two from at least $c'd^2$ vertices in $L_1'$. We work on the event that every $v\in V(G)$ is at distance (in $G$) at most two from at least $c'd^2$ vertices in $L_1'$.

Throughout the proof, we will introduce constants $c_1$ up to $c_8$, under the assumption that each $c_i$ is sufficiently small in terms of $\delta$ and all $c_j$ with $j <i$. 

By assumption, every $v\in V$ is at distance at most two from at least $c'd^2$ vertices in $L_1'$. Let us now define four sets inductively:
\begin{align*}
    &A_1\coloneqq \left\{v\in V\setminus (B\cup A)\colon d_A(v)\ge \frac{c'd}{10}\right\}, \\
    &B_1\coloneqq \left\{v\in V\setminus (B\cup A\cup A_1)\colon d_B(v)\ge \frac{c'd}{10}\right\},\\
    &A_2\coloneqq \left\{v\in V\setminus (B\cup A\cup A_1 \cup B_1)\colon d_{A_1}(v)\ge \frac{c'd}{10}\right\},\\ &B_2\coloneqq \left\{v\in V\setminus(B\cup A\cup A_1 \cup B_1 \cup A_2)\colon d_{B_1}(v)\ge \frac{c'd}{10}\right\}.
\end{align*}
\begin{figure}[H]
\centering
\includegraphics[width=0.7\textwidth]{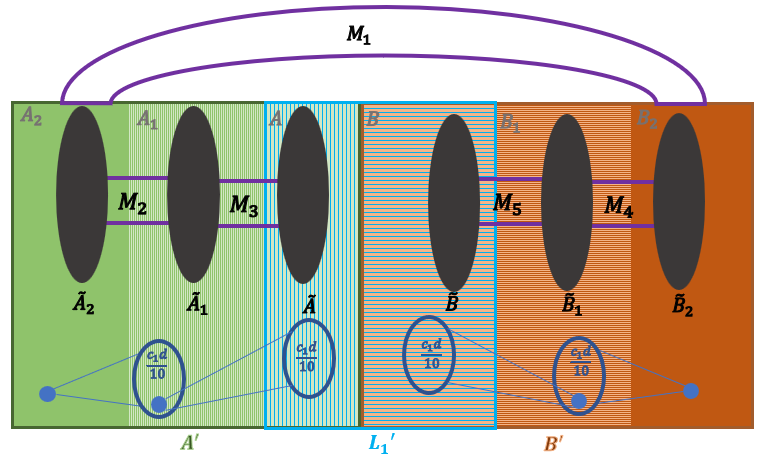}
\caption{An illustration of the sets and matchings in Lemma \ref{l: disjoint paths}. The matchings $M_1$ through $M_5$, in purple, are ordered according to the order they are constructed in the proof. In dark blue, one can see the properties of vertices in $A_2, B_2, A_1$ and $B_1$, with respect to their set of neighbours in $A_1, B_1, A$ and $B$, respectively. Observe that if the first matching $M_1$ had many endpoints in $A'\setminus A_2$ (or $B'\setminus B_2$), we could continue in the same manner with fewer matchings required.}
\label{f: matchings}
\end{figure}

Let us set $A'=A\cup A_1\cup A_2$, and $B'=B\cup B_1\cup B_2$. Observe that $V=A'\sqcup B'$. Indeed, it is clear by the definition of the sets that $A'\cap B'=\varnothing$. Suppose towards contradiction that $v\notin A'\cup B'$, and let us consider the number of vertices in $L_1'$ that are the endpoints of paths of length at most two starting from $v$. There are at most $d$ vertices in $L_1'$ that are neighbours of $v$. As for paths of length exactly two, they are of the form $vux$. If $u\in A_1$, then since $v\notin A'$, and in particular $v\notin A_2$, we have at most $\frac{c'd^2}{10}$ possible choices for the path. Similarly, if $u\in B_1$, then since $v\notin B'$, and in particular $v\notin B_2$, we have at most $\frac{c'd^2}{10}$ possible choices for the path. Finally, if $u\notin A_1\cup B_1$, since $v\notin A_1\cup B_1$, we have at most $\frac{c'd^2}{5}$ possible choices for the path. Altogether, we have at most $\frac{2c'd^2}{5}+d<c'd^2$ vertices in $L_1'$ that are at distance at most two from $v$ --- a contradiction.

Since $A'\sqcup B'=V$, by Theorem \ref{th: iso 2},
\begin{align*}
    e(A',B')\ge \frac{k}{C-1} \log_C \left(\frac{n}{k}\right)\ge c_1k\ln\left(\frac{n}{k}\right)\eqqcolon s.
\end{align*}
By Lemma \ref{l: matching}, with probability at least $1-\exp\left(-\frac{c_2s}{d}\right)$, there exists a matching of size at least $\frac{c_2s}{d}$ between $A'$ and $B'$ in $G_{p_2}$. We continue under the assumption that at least $\frac{c_2s}{3d}$ of the edges in the matching have endpoints in both $A_2$ and $B_2$, as the other cases follow more easily, with fewer matching edges required (see Figure \ref{f: matchings} for an illustration). Let us denote these endpoints of the matching by $\tilde{A}_2$ and $\tilde{B}_2$, respectively. 

Now, every $v\in A_2$, and in particular in $\tilde{A}_2$, has at least $\frac{c'd}{10}$ neighbours in $A_1$. Hence, with probability at least $1-\exp\left(-\frac{c_2s}{d}\right)$ we have a set of at least $\frac{c_2s}{3d}\cdot \frac{c'd}{10}=c_3s$ edges between $\tilde{A}_2$ and $A_1$. Thus, by Lemma \ref{l: matching}, with probability at least $1-\exp\left(-\frac{c_4s}{d}\right)$ there exists a matching of size at least $\frac{c_4s}{d}$ between $\tilde{A}_2$ and $A_1$. Denote by $\tilde{A}_2$ and $\tilde{A}_1$ the corresponding vertices in $\tilde{A}_2$ and $A_1$ of this matching. Since every $v\in A_1$, and in particular in $\tilde{A}_1$, has at least $\frac{c'd}{10}$ neighbours in $A$, with probability at least $1-\exp\left(-\frac{c_4s}{d}\right)$ there are at least $\frac{c_4s}{d}\cdot \frac{c'd}{10}=c_5s$ edges between $\tilde{A}_1$ and $A$. Once again, by Lemma \ref{l: matching}, with probability at least $1-\exp\left(-\frac{c_6s}{d}\right)$, there exists a matching of size at least $\frac{c_6s}{d}$ between $\tilde{A}_1$ and $A$. Denote the endpoints of this matching in $A$ by $\tilde{A}$. Altogether, we obtain with probability at least $1-\exp\left(-\frac{c_{7}s}{d}\right)$ a family of at least $\frac{c_7s}{d}$ vertex-disjoint paths of length three between $\tilde{B}_2\subseteq B_2$ and $\tilde{A}\subseteq A$.

Working similarly in $B'$, we define $\tilde{B}_1\subseteq B_1$ and $\tilde{B}\subseteq B$, and find with probability at least $1-\exp\left(-\frac{c_8s}{d}\right)$ a family of at least $\frac{c_8s}{d}$ vertex-disjoint paths of length at most five, starting from $\tilde{A}\subseteq A$, going through $\tilde{A}_1\subseteq A_1$, $\tilde{A}_2\subseteq A_2$, $\tilde{B}_2\subseteq B_2$, and $\tilde{B}_1\subseteq B_1$ to $\tilde{B}\subseteq B$ (see Figure \ref{f: matchings} for an illustration). Choosing $c\le c_8$ completes the proof.
\end{proof}

The following lemma is then key to the proof of Theorem \ref{th: expansion}. We effectively enumerate the number of subsets of $L'_1$ which do not expand well using a novel double-counting argument to enumerate them in terms of their boundaries, which by assumption are significantly smaller than the sets themselves. This allows us to apply the probability bound from Lemma \ref{l: disjoint paths} to conclude that \textbf{whp} all subsets of $L'_1$ expand relatively well after sprinkling.
\begin{lemma}\label{l: key lemma}
There exists a constant $c=c(\delta)>0$ such that \textbf{whp} for any $S\subseteq V(L_1')$ the following hold.
\begin{enumerate}[(a)]
    \item\label{i:direct neighbourhood} If $\frac{n}{d}\le |S|\le \frac{3\epsilon n}{2}$, then either
    \begin{align*}
        |N_{L_1'}(S)|\ge \frac{c|S|}{d\ln d},
    \end{align*}
    or there exists a family of at least $\frac{c|S|}{d}$ vertex disjoint paths of length at most five between $S$ and $V(L_1')\setminus S$ in $G_{p_2}$;
    \item\label{i:nondirect neighbourhood} 
    If $|S|=\omega(d)$ and $|S|\le \frac{3\epsilon n}{2}$, then either
    \begin{align*}
        |\partial_{L_1'}(S)|\ge \frac{c|S|\ln\left(\frac{n}{|S|}\right)}{d\ln d},   
    \end{align*}    
    or there exists a family of at least $\frac{c|S|}{d}$ vertex disjoint paths of length at most five between $S$ and $V(L_1')\setminus S$ in $G_{p_2}$.
\end{enumerate}
\end{lemma}
We note that the assumption that $|S|=\omega(d)$ in \ref{l: key lemma}\ref{i:nondirect neighbourhood} can be strengthened, however it suffices for our usage and allows for a simpler proof.
\begin{proof}
We argue via two-round exposure, beginning by exposing $G_{p_1}$. By Lemma \ref{l: density}, \textbf{whp} every $v\in V(G)$ is at distance at most two from at least $c_1d^2$ vertices in $L_1'$, for some $c_1=c_1(\epsilon,\delta)>0$. Finally, by Lemma \ref{l: high degree}, \textbf{whp} there are at most $\frac{n}{d^4}$ vertices with degree larger than $\ln d$. We continue assuming that these properties hold deterministically.

We begin with part \ref{i:direct neighbourhood}. Given $\frac{n}{d}\le |S|\le \frac{3\epsilon n}{2}$, let $k\coloneqq |S|$ and let $b_1\coloneqq |N_{L_1'}(S)|$. As we aim to bound the expansion of the set $S$, we may assume that $b_1< \frac{ck}{d\ln d}$, as otherwise the claim holds. In order to facilitate a union bound argument, let us estimate the number of subsets $S$ of size $k$ in $L_1'$ such that $|N_{L_1'}(S)|=b_1$. Let $e_1 = \partial_{L_1'}(N_{L_1'}(S))$. Since there are at most $\frac{n}{d^4}$ vertices with degree larger than $\ln d$, we have that $e_1\le \frac{n}{d^3}+b_1\ln d$. Since $L_1'$ is connected, there are at most $e_1+1$ components in $L_1'\setminus N_{L_1'}(S)$. Furthermore, since $S$ has no neighbours outside $N_{L_1'}(S)$, it must be the union of components in $L_1'\setminus N_{L_1'}(S)$. Hence, the number of ways to choose such an $S$ is at most $\binom{n}{b_1}\cdot 2^{e_1+1}$. Thus, there are at most
\begin{align*}
    \sum_{b_1=1}^{\frac{ck}{d\ln d}}\binom{n}{b_1}2^{\frac{n}{d^3}+b_1\ln d+1}&\le \left(\frac{en}{\frac{ck}{d\ln d}}\right)^{\frac{ck}{d\ln d}}2^{\frac{2ck}{d}}\\
    &\le \exp\left(\frac{ck}{d\ln d}\left(\ln\left(\frac{end\ln d}{ck}\right)+2\ln d\right)\right)\\
    &\le\exp\left(\frac{2ck}{d\ln d}\left(\ln\left(\frac{n}{k}\right)+2\ln d\right)\right)\le \exp\left(\frac{6ck}{d}\right)
\end{align*}
sets $S\subseteq V(L_1')$ with $|N_{L_1'}(S)|<\frac{ck}{d\ln d}$, where we used the fact that $k\ge \frac{n}{d}$ in the first and last inequalities.

We now turn to facilitate a union bound argument for part \ref{i:nondirect neighbourhood}. Given $S\subseteq V(L_1')$ with $|S|=k\le \frac{3\epsilon n}{2}$, we may assume that $|\partial_{L_1'}(S)|<\frac{ck\ln\left(\frac{n}{k}\right)}{d\ln d}$, as otherwise the claim holds. Let us then estimate the number of sets $S$ such that $|\partial_{L_1'}(S)|<\frac{ck\ln\left(\frac{n}{k}\right)}{d\ln d}$.

Let $e_2\coloneqq|\partial_{L_1'}(S)|<\frac{ck\ln\left(\frac{n}{k}\right)}{d\ln d}$ and let $b_1\coloneqq |N_{L_1'}(S)|$ as before. If we write $m$ for the number of components in $G_{p_1}[S]$, then, since $L_1'$ is connected, $m \leq e_2+1$. Moreover, since $S$ has no neighbours outside $N_{L_1'}(S)$, it must be the union of precisely $m$ components of $L'_1 \setminus N_{L_1'}(S)$.

Hence, since $L'_1 \setminus N_{L_1'}(S)$ has at most $n$ components, the number of ways to choose such an $S$ is at most $\binom{n}{b_1}\binom{n}{m}$. Thus, since $b_1 \leq e_2$, there are at most
\begin{align*}
    \sum_{b_1=1}^{e_2}\sum_{m=1}^{e_2+1}\binom{n}{b_1}\binom{n}{m}&\le \left(\frac{en}{\frac{ck\ln\left(\frac{n}{k}\right)}{d\ln d}}\right)^{2\frac{ck\ln\left(\frac{n}{k}\right)}{d\ln d}}\le \exp\left(\frac{2ck\ln\left(\frac{n}{k}\right)}{d\ln d}\ln\left(\frac{end\ln d}{ck\ln\left(\frac{n}{k}\right)}\right)\right)\\
    &\le\exp\left(\frac{2ck}{d}\cdot\frac{\ln\left(\frac{n}{k}\right)}{\ln d}\cdot \left(\ln\left(\frac{n}{k}\right)+2\ln\left(\frac{d\ln d}{\ln\left(\frac{n}{k}\right)}\right)\right)\right)\\
    &\le \exp\left(\frac{3ck}{d}\right)
\end{align*}
sets $S\subseteq V(L_1')$ with $|\partial_{L_1'}(S)|<\frac{ck\ln\left(\frac{n}{k}\right)}{d\ln d}$.

Fix $S\subseteq V(L_1')$ with $|S|=k$. By Lemma \ref{l: disjoint paths}, with probability at least $1-\exp\left(-\frac{c_1k\ln\left(\frac{n}{k}\right)}{d}\right)$, there exists a family of at least $\frac{c_1k\ln\left(\frac{n}{k}\right)}{d}$ vertex disjoint paths of length at most five between $S$ and $V(L_1')\setminus S$ in $G_{p_2}$, where $c_1$ is the constant from Lemma \ref{l: disjoint paths}. We note that we used our assumption that every $v\in V(G)$ is at distance at most two from at least $c'd^2$ vertices in $L_1'$ in order to invoke Lemma \ref{l: disjoint paths}.

Recalling that $k\le \frac{3\epsilon n}{2}$, the probability there is a set $S$ violating the statement of part \ref{i:direct neighbourhood} is then at most
\begin{align*}
    \exp\left(\frac{6ck}{d}-\frac{c_1k\ln\left(\frac{n}{k}\right)}{d}\right)\le\exp\left(\frac{k}{d}\left(6c-c_1\right)\right).
\end{align*}
Once again, the probability that there is a set $S$ violating the statement of part \ref{i:nondirect neighbourhood} is at most
\begin{align*}
       \exp\left(\frac{3ck}{d}-\frac{c_1k}{d}\right)&=\exp\left(\frac{k}{d}\left(3c-c_1\right)\right).
\end{align*}
Under our assumption that $k\coloneqq |S|=\omega(d)$ and for $c$ small enough with respect to $c_1$, by the union bound the probability having a set $S$ violating the statement of part \ref{i:direct neighbourhood} or \ref{i:nondirect neighbourhood} is $o(1)$.
\end{proof}

\subsection{Structure of subsets in the residue}
As we mentioned, we also require some control over the typical structure of subsets in the residue $L_1 - L'_1$ and their likely expansion into the early giant after sprinkling. Let us begin with the following lemma, showing how the vertices in $L_1'$ are embedded in $L_1$, which generalises Lemma 3.2 in \cite{EKK22}. Given a vertex $v\in V(L_1')$, let $C_v$ be the set of vertices which are contained in components of $L_1-L_1'$, such that there is a vertex adjacent to $v$ in $G_2$ in these components. Also, given a subset $S\subseteq L_1'$, we denote by $C_{S}=\cup_{v\in S}C_v$.

\begin{lemma}\label{l: Kd}
There exists a constant $K_2\coloneqq K_2(C,\epsilon)>0$ such that \textbf{whp} $|C_v|\le K_2d$ for every $v\in V(L_1')$.
\end{lemma}
\begin{proof}
Note that $G_2$ has the same distribution as $G_p$, and that $p_1=\frac{1+\epsilon-\delta+o(1)}{d}$. Furthermore, observe that by Theorem \ref{th: dekk22}, there exists a constant $K_1\coloneqq K_1(C,\epsilon)$ such that \textbf{whp} every component of $G_{p_1}$, besides $L_1'$, is of order at most $K_1d$ (although technically the $K_1$ given by Theorem \ref{th: dekk22} might depend on $\delta$, it is easy to check from the proof that, since $\delta \ll \epsilon$, we may choose $K_1$ only as a function of $\epsilon$ and $C$). 

Suppose that there is some $v\in V(L_1')$ such that $|C_v|\ge K_2d$. Note that $C_v\cup \{v\}$ is connected in $G_2$, and that $C_v$ is the disjoint union of some sets $\left\{C_1,\ldots, C_r\right\}$ where $C_i$ is the vertex set of some component of $G_1$, each of which has order at most $K_1d$. It follows there must be some subset $\hat{C}\subseteq C_v$ such that $\hat{C}\cup\{v\}$ is connected in $G_2$, $\hat{C}$ is the union of some subset of $\{C_1,\ldots, C_r\}$ and $K_2d\le |\hat{C}|\le (K_2+K_1)d$.

In particular, there is some spanning tree $T$ of $\hat{C}\cup\{v\}$, all of whose edges are in $G_2$, and no edge in the edge-boundary of $V(T)\setminus\{v\}$ is present in $G_1$. 

Let us bound the probability that such a tree of order $k$ exists in $G_2$ for each $K_2d+1\le k \le (K_2+K_1)d+1$. By Lemma \ref{l: trees}, there are at most $n(ed)^{k-1}$ such trees. A spanning tree $T$ has $k-1$ edges in $G_2$, which happens with probability at most $p^{k-1}$. Furthermore, since $|V(T)\setminus\{v\}|=k-1$, by Theorem \ref{th: iso 1} there are at least $(k-1)\left(d-(C-1)\log_2(k-1)\right)$ edges in the edge-boundary of $V(T)\setminus \{v\}$, none of which are in $G_1$, which happens with probability at most ${(1-p_1)^{(k-1)\left(d-(C-1)\log_2(k-1)\right)}}$. Whilst these two events are not necessarily independent, they are negatively correlated. Thus, by the union bound, the probability that such a tree of order $k$ exists in $G_2$ is at most
\begin{align*}
    n(ed)^{k-1}p^{k-1}(1-p_1)^{(k-1)\left(d-(C-1)\log_2(k-1)\right)}.
\end{align*}

Therefore, the probability that such a tree exists for $k\in I\coloneqq [K_2d+1, (K_2+K_1)d+1]$ is at most
\begin{align*}
    n\sum_{k\in I}\exp\left((k-1)\left(1+\ln(1+\epsilon)-(1+\epsilon-2\epsilon^3)\right)\right)\le n\sum_{k\in I}\exp\left(-\epsilon^3(k-1)\right)=o(1),
\end{align*}
where we used the fact that $\ln(1+\epsilon)\le \epsilon-3\epsilon^3$ for small enough $\epsilon>0$, and we assume that $K_2\ge \frac{2\ln C}{\epsilon^3}$, recalling that $n\le C^t\le C^d$ and hence $\ln n\le \ln C \cdot d$.
\end{proof}

In order to obtain our results, we will require further information about the likely expansion of subsets in the residue into the early giant. We will require the following density lemma.
\begin{lemma}[Lemma 4.6 of \cite{DEKK22}, rephrased]\label{l: density2}
There exists a constant $c_2>0$ such that for any fixed constants $K,c_1>0$, \textbf{whp} every subset $M \subseteq V(G)$, with $|M|=Kd$ and $G[M]$ connected, contains at most $c_1d$ vertices $v \in M$ such that $|N_G(v)\cap V(L_1')|<c_2d$.
\end{lemma}

We will make use of the following probabilistic lemma, which utilises Lemma \ref{l: tree decomposition}.
\begin{lemma}\label{l: middle sets lemma}
There exist positive constants $K, K'\coloneqq K'(K)$ and $c\coloneqq c(\delta)$ such that the following holds. Let $S\subseteq V(L_1')$ and $B\subseteq V(G)\setminus V(L_1')$ be such that $|S\cup B|\ge Kd$ and $G[S\cup B]$ is connected and $|B|\ge K'|S|$. Then, there exists a matching in $G_{p_2}$ of size at least $c|B|$ between $B$ and $V(L_1')\setminus S$ with probability at least $1-\exp\left(-c|B|\right)$.
\end{lemma}
\begin{proof}
By Lemma \ref{l: density2}, there exists a constant $c_2>0$ such that for any fixed constants $K,c_1>0$, \textbf{whp} every subset $M \subseteq V(G)$, with $|M|=Kd$ and $G[M]$ connected, contains at most $c_1d$ vertices $v \in M$ such that $|N_G(v)\cap V(L_1')|<c_2d$. Let us fix some $K\gg c_1$ and continue assuming the above holds deterministically for the corresponding $c_2$.

Since $G[S\cup B]$ is connected, it has a spanning tree $T$. By Lemma \ref{l: tree decomposition}, applied with $\ell = Kd$, there exist subsets $A_1,\ldots, A_s\subseteq V(T)$ satisfying properties \ref{i: partition}--\ref{i: sizes} of that lemma. In particular, since for all $i\in [s]$ we have $Kd \leq |A_i|\le 3Kd$, by Theorem \ref{th: iso 1} we have that $e(A_i) \le (C-1)3Kd\log_2(3Kd)\le 6CKd\log_2d$. Thus, by our assumption, for all $i\in [s]$ we have that 
\begin{align*}
    e_G(A_i, L_1'\setminus A_i)\ge (Kd-c_1d)c_2d-12CKd\log_2d.
\end{align*} 
Thus, defining $\hat{A}_i\coloneqq A_i\setminus \left(\bigcup_{j\in ([s]\setminus \{i\})} A_j\right)$, we have that $e(\hat{A}_i, L_1'\setminus \hat{A}_i)\ge e(A_i, L_1'\setminus A_i)-d$, and the edge sets $E(\hat{A}_1, L_1'\setminus\hat{A}_1), \ldots, E(\hat{A}_s, L_1'\setminus\hat{A}_s)$ are disjoint. Hence, since $K$ is sufficiently large with respect to $c_1$, we can choose $c_3\coloneqq c_3(c_1,c_2)>0$ small enough such that
\begin{align*}
    e\left(S\cup B, L_1'\setminus (S\cup B)\right)\ge \frac{|S|+|B|}{3Kd}\left((Kd-c_1d)c_2d-12CKd\log_2d\right)-\frac{|S|+|B|}{Kd}\cdot d\ge \left(|S|+|B|\right)c_3d.
\end{align*}
Therefore, as long as $K' \coloneqq K'(c_1,c_2)$ is large enough, there exists $c_4 \coloneqq c_4(c_1,c_2) > 0$ such that
\begin{align*}
    e(B, L_1'\setminus S)\ge |B|c_3d-2|S|d\ge \left(c_3 - \frac{2}{K'} \right)|B|d  \ge c_4|B|d.
\end{align*}
Thus, by Lemma \ref{l: matching}, there exists a constant $c(\delta)>0$ such that with probability at least $1-\exp\left(-c|B|\right)$ there exists a matching $M$ in $G_{p_2}$ of size at least $c|B|$ between $B$ and $L_1'\setminus S$.
\end{proof}

From Lemmas \ref{l: middle sets lemma} and \ref{l: density2}, we can derive the following statement, complementing Lemma \ref{l: key lemma}. 
\begin{lemma}\label{l: key lemma 2}
There exist constants $K, K', c>0$ such that \textbf{whp}, for every $S_1\subseteq V(L_1'), S_1\neq \varnothing$, and for every $S_2\subseteq C_{S_1}$ such that $|S_2|\ge K'|S_1|$, $|S_1\cup S_2|\ge Kd$ and $G[S_1\cup S_2]$ connected, the following holds. Either
\begin{align*}
    |N_{L_1'}(S_1)|\ge \frac{c|S_2|}{d}, \quad \text{or} \quad |N_{G_2}(S_1\cup S_2)|\ge \frac{c|S_2|}{d}.
\end{align*}
\end{lemma}
\begin{proof}
We begin by exposing $G_{p_1}$, and let us fix $\varnothing \neq S_1\subseteq V(L_1')$. Let us now expose all the edges in $G_{p_2}$ which are either inside $V(G)\setminus V(L_1')$ or lie between $S_1$ and $V(G)\setminus V(L_1')$. Denote by $G_1'$ the graph $G_1$ together with these edges, noting that $G_1 \subseteq G'_1 \subseteq G_2$ and that $G'_1$ determines $C_{S_1}$.

Let us choose $c_1>0$ a small enough constant, and choose $K$ a sufficiently large constant. Then by Lemma \ref{l: density2} there exists $c_2>0$ such that \textbf{whp} every connected subset $M\subseteq V(G)$ of size $Kd$ has at most $c_1d$ vertices with less than $c_2d$ neighbours in $L_1'$. Furthermore, by Lemma \ref{l: middle sets lemma} there exist constants $K',c'>0$ (from Lemma \ref{l: density2}) that the conclusion of the lemma holds for $K,c_1,c_2$ and $S_1$, noting that the event in the lemma depends only on edges in $G_{p_2}$ between $C_{S_1}$ and $V(L_1')\setminus S_1$, which we have not yet exposed. We further note that we may choose $K'$ sufficiently large. We continue assuming these properties hold deterministically.

Let us fix $S_2\subseteq C_{S_1}$ satisfying the conditions of the lemma and let $k_1\coloneqq |S_1|$ and $k_2\coloneqq |S_2|$. By Lemma \ref{l: middle sets lemma} the probability that $S_2$ has less than $c'|S_2|$ neighbours in $V(L_1')\setminus S_1$ in $G_{p_2}$ is at most $\exp\left(-c'|S_2|\right)$. Furthermore, the event that $S_2$ has at least $c'|S_2|$ neighbours in $V(L_1)\setminus S_1$ in $G_{p_2}$ clearly implies that $|N_{G_2}(S_1\cup S_2)|\ge \frac{c|S_2|}{d}$, for any constant $c>0$.

Let us now facilitate a union bound argument. Let us choose $c\coloneqq c(C,c')$ sufficiently small and suppose that $b_1\coloneqq |N_{L_1'}(S_1)|<\frac{ck_2}{d}$, as otherwise the claim holds. Let us further fix $k_2$ for now. Let us write $m$ for the number of components in $G_{1}\setminus N_{L_1'}(S_1)$. Since $L_1'$ is connected and $G$ is $d$-regular, we have $m\le d\cdot b_1+1$. 

Hence, since $S_1$ has no neighbours in $G_1$ outside $N_{L_1'}(S_1)$, it must be the union of some components of $G_{1}\setminus N_{L_1'}(S_1)$, and so the number of ways to choose such an $S_1$ is at most $\binom{n}{b_1}2^{m}$. Thus, there are at most
\begin{align*}
    \sum_{b_1=1}^{\frac{ck_2}{d}}\sum_{m=1}^{ck_2}\binom{n}{b_1}2^{m}
    &\le \left(\frac{en}{\frac{ck_2}{d}}\right)^{\frac{ck_2}{d}}\cdot2^{ck_2+1}\le \exp\left(\frac{ck_2}{d}\left(\ln\left(\frac{end}{ck_2}\right)+2d\right)\right)\le \exp\left(5\ln C\cdot ck_2\right)
\end{align*}
sets $S_1\subseteq V(L_1')$ with $|N_{L_1'}(S_1)|<\frac{ck_2}{d}$, where we used the assumption that $\ln n\le \ln C \cdot d$ and that $k_2\ge d$, since we may choose $K\ge 2$ and $K'$ large enough.

Now, let us consider the number of ways to choose $S_2\subseteq C_{S_1}$, noting that having determined $S_1$, choosing $S_1\cup S_2$ determines $S_2$. We may assume that $b_2\coloneqq|N_{G_1'}(S_1\cup S_2)|\le \frac{ck_2}{d}$, since $G_1'\subseteq G_2$. Since $G_1'[S_1\cup S_2]$ is connected, and has all its neighbours in $N_{G_1'}(S_1\cup S_2)$, exactly one of the at most $n$ components in $G_1'\setminus N_{G_1'}(S_1\cup S_2)$ is $S_1\cup S_2$. Since $S_1$ is fixed, we can identify this component. Hence, the number of ways to choose $S_1\cup S_2$ with $|N_{G_1'}(S_1\cup S_2)|\le \frac{ck_2}{d}$ is at most the number of ways to choose a set of size at most $\frac{ck_2}{d}$ in $V(G_1')$. 
That is at most
\begin{align*}
   \sum_{b_2=1}^{\frac{ck_2}{d}}\binom{n}{b_2}\le \left(\frac{en}{\frac{ck_2}{d}}\right)^{\frac{ck_2}{d}}\le \exp\left(4\ln C\cdot ck_2\right).
\end{align*}

Therefore, for fixed $k_2$, the probability of an event violating the statement of the lemma is at most
\begin{align*}
    \exp\left(9\ln C\cdot ck_2\right)\exp\left(-c'k_2\right)=o(1/n),
\end{align*}
for $c$ small enough in terms of $C$ and $c'$. Union bound over the at most $n$ choices of $k_2$ completes the proof.
\end{proof}

\subsection{Proof of Theorems \ref{th: connected expansion}\ref{i:medium} and \ref{th: expansion}\ref{i:large}}
\begin{proof}[Proof of Theorem \ref{th: connected expansion}\ref{i:medium} and \ref{th: expansion}\ref{i:large}]
Let $S_1=S\cap V(L_1')$ and $S_2=S\cap \left(V(L_1)\setminus V(L_1')\right)$. Let $c_{\ref{l: key lemma}}$ be the constant whose existence is asserted in Lemma \ref{l: key lemma}, and let $K_{\ref{l: key lemma 2}}, K'_{\ref{l: key lemma 2}}$ and $c_{\ref{l: key lemma 2}}$ be the constants whose existence is asserted in Lemma \ref{l: key lemma 2}. Let $c>0$ be sufficiently small in terms of $c_{\ref{l: key lemma}}$, $K'_{\ref{l: key lemma 2}}$ and $c_{\ref{l: key lemma 2}}$.
\begin{enumerate}
    \item[\textit{\ref{th: connected expansion}\ref{i:medium}}] Recall that we assume that $K_{\ref{l: key lemma 2}}d \leq n^{\epsilon^5}\le |S|\le \frac{3\epsilon n}{2}$ and $G_p[S]$ is connected.

   Suppose $|S_1|\ge \frac{|S_2|}{K'_{\ref{l: key lemma 2}}}$. Then, $|S_1|=\Omega(d\ln d)=\omega(d)$ and so by Lemma \ref{l: key lemma}\ref{i:nondirect neighbourhood} \textbf{whp} either $$|\partial_{L_1'}(S_1)|\ge \frac{c_{\ref{l: key lemma}}|S_1|\ln\left(\frac{n}{|S_1|}\right)}{d\ln d}\ge\frac{c|S|\ln\left(\frac{n}{k}\right)}{d\ln d},$$ or there is a family of at least $\frac{c_{\ref{l: key lemma}}|S_1|}{d}\ge \frac{c|S|}{d}$ vertex-disjoint paths from $S_1$ to $V(L_1')\setminus S_1\subseteq V(L_1)\setminus S$. However, since each such path contributes a unique vertex to the neighbourhood of $S$ in $V(L_1)$ (the first vertex along the path which is not in $S$), in the latter case $|N_{G_p}(S)|\ge \frac{c|S|}{d}$, and so the result follows.

    Otherwise, $|S_2|\ge K'_{\ref{l: key lemma 2}}|S_1|$ and so by Lemma \ref{l: key lemma 2} \textbf{whp} either $|N_{L_1'}(S_1)|\ge \frac{c_{\ref{l: key lemma 2}}|S_2|}{d}$, or $|N_{L_1}(S)|\ge \frac{c_{\ref{l: key lemma 2}}|S_2|}{d}$. In the first case, $|\partial_{L_1}(S)|\ge |N_{L_1'}(S_1)|\ge \frac{c|S|}{d}$ and, similarly to before, in the second case $|\partial_{L_1}(S)|\ge |N_{L_1}(S)|\ge \frac{c|S|}{d}$. 

    \item[\textit{\ref{th: expansion}\ref{i:large}}] We now assume that $K_{\ref{l: key lemma 2}}d \leq \epsilon ^2 n\le |S|\le \frac{3\epsilon n}{2}$. Note that, since $\big||V(L_1)|-|V(L_1')|\big|\le 4\epsilon^3n$, it follows that $|S_1|\ge \frac{2|S|}{3}$. Thus, by Lemma \ref{l: key lemma}\ref{i:direct neighbourhood}, \textbf{whp} either $$|N_{L_1'}(S_1)|\ge \frac{c_{\ref{l: key lemma}}|S_1|}{d\ln d}\ge \frac{c|S|}{d\ln d},$$ or there is a family of at least $\frac{c_{\ref{l: key lemma}}|S_1|}{d}\ge\frac{c|S|}{d}$ vertex-disjoint paths from $S_1$ to $V(L_1')\setminus S_1 \subseteq V(L_1) \setminus S$, and each such path contributes a unique vertex to the neighbourhood of $S$ in $L_1$. As before, in either case $|N_{G_p}(S)|\ge \frac{c|S|}{d\ln d}$. 
\end{enumerate}
\end{proof}

The proof of Theorem \ref{th: expander} will follow from key ideas from \cite{K19} together with our expansion result on large sets (Theorem \ref{th: expansion}\ref{i:large}).
\begin{proof}[Proof of Theorem \ref{th: expander}]
Let $c$ be the constant whose existence is asserted in Theorem \ref{th: connected expansion}. Let $M\subseteq V(L_1)$ be a maximal set such that $|M|\le \frac{\epsilon n}{10}$ and $|N_{G_p}(M)|\le \frac{c|M|}{d\ln d}$. Let $H=L_1- M$. Assume that there is some subset $B\subseteq V(H)$ such that $|B|\le \frac{|V(H)|}{2}$ and $|N_{H}(B)|\le \frac{c|B|}{d\ln d}$. Then,
\begin{align*}
    |N_{G_p}(M\cup B)| \le |N_{G_p}(M)| + |N_{H}(B)| <\frac{c|M|}{d\ln d}+\frac{c|B|}{d\ln d}=\frac{c|M\cup B|}{d\ln d}.
\end{align*}
Thus, by the maximality of $M$, we obtain that $|M\cup B|\ge \frac{\epsilon n}{10}$. However, by Theorem \ref{th: expansion}\ref{i:large}, every subset $S\subseteq V(L_1)$ with $\epsilon^2n\le |S|\le \frac{3\epsilon n}{2}$ has $|N_{G_p}(S)|\ge \frac{c|S|}{d\ln d}$. Hence, $|M\cup B|\ge \frac{3\epsilon n}{2}$. 

On the other hand, by our choice of $B$ and $M$, we have that
\begin{align*}
    |M\cup B|\le |M|+\frac{|V(L_1)|-|M|}{2}= \frac{|V(L_1)|+|M|}{2}\le \frac{|V(L_1)|}{2}+\frac{\epsilon n}{20}.
\end{align*}
By Theorem \ref{th: dekk22}, \textbf{whp} $|V(L_1)|\le 2\epsilon n$, and hence $|M\cup B|\le \frac{21\epsilon n}{20}<\frac{3\epsilon n}{2}$ --- a contradiction. Hence, \textbf{whp} $H$ has the desired expansion properties. Furthermore, by Theorem \ref{th: dekk22} \textbf{whp}
\begin{align*}
    |V(H)|=|V(L_1)|-|M|\ge \frac{19\epsilon n}{10}-\frac{\epsilon n}{10}\ge \frac{3\epsilon n}{2}.
\end{align*}
\end{proof}

\section{Consequences of expansion in the giant component}\label{s: consequence}
We begin with the likely existence of a long cycle, which follows immediately from Theorem \ref{th: expansion}\ref{i:large} together with Theorem \ref{th: krivelevich cycle}.
\begin{proof}[Proof of Theorem \ref{th: consequences}\ref{i: cycle}]
By Theorem \ref{th: expansion}\ref{i:large},  there exists a constant $c<0$ such that \textbf{whp} for all $\epsilon^2n\le k\le \frac{3\epsilon n}{2}$ and all subsets $S\subseteq V(L_1)$ with $|S|=k$,
\begin{align*}
    |N_{G_p}(S)|\ge \frac{c|S|}{d\ln d}.
\end{align*}
Thus, applying Theorem \ref{th: krivelevich cycle} with $a=\frac{3\epsilon n}{2}$ and $b=\frac{3c\epsilon n}{4d\ln d}$, we obtain that \textbf{whp} $L_1$ contains a cycle of length $\Omega\left(\frac{n}{d\ln d}\right)$.
\end{proof}

Note that, due to the comment after Theorem \ref{th: expander}, up to a logarithmic factor in $d$ this is the best bound that can be given with such an argument based solely on the expansion of $L_1$. 

We now turn to Theorem \ref{th: consequences}\ref{i: diameter} and \ref{i: mixing time}. For these two theorems, the following two lemmas will be useful. The first is a variant of a lemma from \cite{DK21B}, bounding the typical number of edges incident to connected subsets in $G_p$, whose proof we include for completeness. 
\begin{lemma}\label{l: edges incident to subsets}
\textbf{Whp}, for all $S\subseteq V(L_1)$ such that $G_p[S]$ is connected, 
\begin{align}\label{e:edgesbound}
    e_{G_p}(S)+e_{G_p}(S,S^C)\le \max\left\{10|S|, 20\ln C \cdot d\right\}.
\end{align}
\end{lemma}
\begin{proof}
Let us begin by considering connected sets $S$ such that $|S|=k\ge \ln C \cdot d$. Since any connected set in $G_p$ has a spanning tree, it is sufficient to show that \eqref{e:edgesbound} holds whenever $S$ is the vertex set of a tree in $G_p$ of order $k \geq d$ in $G_p$. By Lemma \ref{l: trees}, there are at most $n(ed)^{k-1}$ trees on $k$ vertices in $G$ and the probability that each such tree is contained $G_p$ is $p^{k-1}$. Since each set of $k$ vertices is incident to at most $kd$ edges in $G$,  there are most $\binom{kd}{9k}$ ways to choose an additional $9k$ edges incident to this set of vertices, and these edges are in $G_p$ with probability $p^{9k}$. Hence, by the union bound, the probability that \eqref{e:edgesbound} fails to hold is at most:
\begin{align*}
    n(ed)^{k-1}p^{k-1}\binom{kd}{9k}p^{9k}&\le n\cdot \left(2e\right)^{k-1}\left(\frac{2e}{9}\right)^{9k}\le n\exp(-2k)=o(1/n),
\end{align*}
since $k\ge \ln C \cdot d  \ge \ln n$. Taking a union bound over the at most $n$ possible values of $k$, it follows that \textbf{whp} for all subsets $S\subseteq V(L_1)$ with $|S|\ge \ln C \cdot d$ and $G_p[S]$ connected, $e(S)+e(S, S^C)\le 10|S|$.

We now turn to connected sets $S$ with $|S|<\ln C\cdot d$. Since $L_1$ is connected, and by Theorem \ref{th: dekk22} we have that \textbf{whp} $|V(L_1)|\ge \epsilon n$, there exists a connected set $S'\supseteq S$ such that $|S'|=d\ln d$. Note that $2e(S)+e(S,S^c)\le 2e(S')+e(S',S'^C)$, and so in particular by the above \textbf{whp}
\begin{align*}
    e(S)+e(S, S^C)\le 2e(S)+e(S,S^C)\le 2\left(e(S')+e(S', S'^C)\right)\le 20\ln C \cdot d,
\end{align*}
completing the proof.
\end{proof}

We also require a bound on the typical number of edges in $L_1$. While this can be calculated quite accurately, the following naive, yet simple to prove bound will suffice for our goals, and utilises the Depth First Search (DFS) algorithm (see \cite{K19a} for definition and application of the DFS algorithm in random graphs). Recall that the excess of a connected graph $H$ is defined as $|E(H)|-(|V(H)|-1)$. 
\begin{lemma}\label{l: e(L_1)}
\textbf{Whp}, $e(L_1)<3\epsilon n$.
\end{lemma}
\begin{proof}
We begin by running a DFS algorithm with $\frac{nd}{2}$ random bits $X_i$, to expose a spanning forest of $G_p$. We first claim that if there is a connected component $S$ of order $k$ with $k\ge d^2$, then we have queried at least $\frac{2kd}{3}$ of the edges incident to $S$. Indeed, otherwise, there would have been an interval of length at most $\frac{2kd}{3}$ where we receive $k$ positive answers. By a typical Chernoff-type bound, the probability that a fixed interval of length $\frac{2kd}{3}$ contains $k$ positive answers is at most
\begin{align*}
    \mathbb{P}\left(Bin\left(\frac{2kd}{3},\frac{1+\epsilon}{d}\right)\ge k\right)\le \exp\left(-\frac{k}{30}\right) \leq \exp\left(-\frac{d^2}{30}\right) .
\end{align*}
In particular, taking a union bound over the at most $nd$ intervals of length $\frac{2kd}{3}$ and at most $n$ different values of $k$ completes the proof of the claim. 

By Theorem \ref{th: dekk22}, \textbf{whp} this algorithm discovered a unique giant component $L_1$, with $|V(L_1)|< 2\epsilon n$, and in doing so queried at least $\frac{2|V(L_1)|d}{3}$ of the at most $|V(L_1)|d$ edges incident to $V(L_1)$. However, since we exposed a spanning tree of $L_1$, at most $|V(L_1)|-1$ edges of $L_1$ were exposed during the algorithm. Since there are at most $\frac{|V(L_1)|d}{3}$ queries left and \textbf{whp} $|V(L_1)<2\epsilon n$, the number of excess edges in $L_1$ is stochastically dominated by a binomial random variable $\text{Bin}\left(\frac{2\epsilon n d}{3}, \frac{1+\epsilon}{d}\right)$. In particular, by a standard Chernoff-type bound, \text{whp} $L_1$ has at most $\epsilon n$ excess edges and hence in total $e(L_1) \leq |V(L_1)| -1 + \epsilon n < 3\epsilon n$.
\end{proof}

\subsection{Proof of Theorem \ref{th: consequences}\ref{i: diameter}}
\begin{proof}[Proof of Theorem \ref{th: consequences}\ref{i: diameter}]
We note that by Theorem \ref{th: dekk22} and Lemma \ref{l: e(L_1)}, \textbf{whp}, $\epsilon n < |E(L_1)|, |V(L_1)|< 3\epsilon n$, and we assume in what follows that this holds. Given a vertex $v \in V(L_1)$, let $B(v,r)$ denote the ball of radius $r$ around $v$ in $L_1$. Since $L_1$ is connected and has size at least $\epsilon n$, for any $v\in V(L_1)$ we have that $|B(v,d\ln d)|\ge d\ln d$. Furthermore, by Lemma \ref{l: edges incident to subsets}, \textbf{whp} for any $B(v,r)\subseteq V(L_1)$ with $|B(v,r)|\ge \ln C \cdot d$, 
\[
\frac{e\left(B(v,r)\right)}{10}  \le |B(v,r)|\le e(B(v,r)) -1,
\]
where the lower bound holds since $B(v,r)$ is connected. By Theorem \ref{th: connected expansion}\ref{i:small} and \ref{i:medium}, \textbf{whp} for any $B(v,r)\subseteq V(L_1)$ with $|B(v,r)|\ge d\ln d$,
\begin{align*}
    e(B(v,r+1))&= e(B(v,r)) + \partial_{G_p}(B(v,r)) \\ 
    &\geq \min\left\{\frac{3\epsilon n}{2}-1,e(B(v,r)) +\frac{c\ln\left(\frac{n}{|B(v,r)|}\right)}{d\ln d}|B(v,r)|\right\}. 
\end{align*}
By the above, \textbf{whp}
\begin{align*}
    e(B(v,r)) +\frac{c\ln\left(\frac{n}{|B(v,r)|}\right)}{d\ln d}|B(v,r)|&\ge \left(1+\frac{c\ln\left(\frac{n}{e(B(v,r))-1}\right)}{10d\ln d}\right)e(B(v,r))\\&\ge \left(1+\frac{c'\ln\left(\frac{n}{e(B(v,r))}\right)}{10d\ln d}\right)e(B(v,r)),
\end{align*}
for some constants $c,c'>0$, and hence \textbf{whp}
\begin{align}
    e(B(v,r+1))\ge \min\left\{\frac{3\epsilon n}{2}-1,\left(1+\frac{c'\ln\left(\frac{n}{e(B(v,r))}\right)}{10d\ln d}\right)e(B(v,r))\right\}. \label{e:edgeexpansion}
\end{align}
We continue assuming the above holds deterministically.

Let $v$ be an arbitrary vertex in $L_1$. We let $B_0\coloneqq B(v,d\ln d)$, and define inductively $B_i\coloneqq B(v,d\ln d +i)$. 

Let $C'>0$ be such that $n=\exp\left(C'd\right)$. Given $\frac{1}{d}<\alpha\le 1$, we define
\begin{align*}
    I(\alpha)\coloneqq \left\{i\in \mathbb{N}\colon \exp\left((1-\alpha)C'd\right)\le e(B_i)\le  \exp\left(\left(1-\frac{\alpha}{2}\right)C'd\right)\right\}.
\end{align*}
Using \eqref{e:edgeexpansion} we can bound the size of $I(\alpha)$. For each $i\in I(\alpha)$, we have that $\frac{c'\ln\left(\frac{n}{e(B(i))}\right)}{10d\ln d}\ge \frac{c' C'\alpha}{20\ln d} : =\frac{c''\alpha}{\ln d}$. Thus by  \eqref{e:edgeexpansion},
\begin{align*}
    |I(\alpha)|\le \log_{1+\frac{c''\alpha}{\ln d}}\left(\frac{\exp\left(\left(1-\frac{\alpha}{2}\right)C'd\right)}{\exp\left((1-\alpha)C'd\right)}\right)=\frac{\alpha C'd}{2\ln\left(1+\frac{c''\alpha}{\ln d}\right)}=O(d\ln d).
\end{align*}
Let $i_{\mathrm{max}}$ be the smallest index such that $e(B_i)>\frac{3\epsilon n}{2}-1$, let $\alpha_0=1$ and let $\alpha_j=\frac{\alpha_0}{2^j}$. Then, there is a smallest index $j_{\mathrm{max}}$ such that 
\begin{align*}
    [i_{\mathrm{max}}]=\bigcup_{j=1}^{j_{\mathrm{max}}}I(\alpha_j).
\end{align*}
Furthermore, there is some constant $C''$ such that if we let $\alpha_{\mathrm{max}}=\frac{C''\ln\left(\frac{1}{\epsilon}\right)}{d}$, then $\exp\left((1-\alpha_{\mathrm{max}})C'd\right)=\frac{3\epsilon n}{2}$.
 Since $\alpha_i=\frac{\alpha_0}{2^{i}}$, it follows that 
 $$j_{\mathrm{max}} \leq \left\lceil \log_2\left(\frac{d}{C''\ln\left(\frac{1}{\epsilon}\right)}\right) \right \rceil=O(\ln d).$$
Thus, 
\[i_{\mathrm{max}} \leq j_{\mathrm{max}} \cdot \max_{j \leq j_{\mathrm{max}}} |I(\alpha_j)| = O(d \ln^2 d).
\]

Therefore it follows that there is some constant $K>0$ such that for every $v\in V(L_1)$,
\[
e\left(B(v,Kd\ln^2d)\right) \geq e\left(B(v,(K-1)d\ln^2d)\right)\ge \frac{3\epsilon n}{2}-1 \geq \frac{|E(L_1)|}{2}.
\]
Since $L_1$ is connected, we have that $e\left(B(v,Kd\ln^2d+1)\right)>\frac{|E(L_1)|}{2}$.

Thus, we can cover more than half of $E(L_1)$ within a ball of radius $O(d\ln^2d)$ from any vertex $v\in V(L_1)$, and therefore the diameter of $L_1$ is $O(d\ln^2d)$.
\end{proof}

\subsection{Proof of Theorem \ref{th: consequences}\ref{i: mixing time}}
We start with some definitions and brief background (see \cite{LPW17} for a more comprehensive introduction to Markov chains and mixing time). Given a graph $G$, the \textit{lazy simple random walk} on $G$ is a Markov chain starting at a vertex $v_0$ chosen according to some distribution $\sigma$, such that for any vertex $v\in V(G)$ the walk stays at $v$ with probability $\frac{1}{2}$, and otherwise moves to a uniformly chosen random neighbour $u$ of $v$. Hence, the transition probability from $v$ to $u$ satisfies $\mathbb{P}(v\to u)=\frac{1}{2d(v)}$. If $G$ is connected, then this Markov chain is irreducible and ergodic and as such has a stationary distribution, which we call the \textit{stationary distribution} $\pi$, which can be seen to be given by $\pi(v)=\frac{d(v)}{2e(G)}$ for each $v\in V(G)$. We are interested in estimating how quickly this Markov chain converges to its limit distribution. For that, recall that the \textit{total variation distance} $d_{TV}$ between two distributions $p_1$ and $p_2$ on $V(G)$ is defined by 
\[
d_{TV}(p_1,p_2):=\max_{A\subset V(G)}\bigg|p_1(A)-p_2(A)\bigg|.
\]
Let $P^t(v,\cdot)$ denote the distribution on $V(G)$ given by starting the lazy random walk at $v\in V(G)$ and running for $t$ steps. If we define $d(t):=\max_{v\in V(G)}d_{TV}\left(P^t(v,\cdot),\pi\right)$, then the \textit{mixing time} of the lazy random walk is then defined as $t_{mix}:=\min\left\{t:d(t)\le \frac{1}{4}\right\}.$
Now, for any $S\subseteq V(G)$, let
\begin{align*}
    \pi(S):=\sum_{v\in S}\pi(v)=\frac{2e(S)+e(S,S^C)}{2e(G)} \qquad \text{ and } \qquad 
    Q(S):=\sum_{v\in S, u\in S^C}\pi(v)\mathbb{P}(v\to u)=\frac{e(S,S^C)}{4e(G)}.
\end{align*}
The \textit{conductance} $\Phi(S)$ of $S$  is then given by
\begin{align*}
    \Phi(S):=\frac{Q(S)}{\pi(S)\pi(S^C)}=\frac{e(S, S^C)}{2\left(2e(S)+e(S,S^C)\right)\pi(S^C)},
\end{align*}
where we note that since $Q(S)=Q(S^C)$, we have that $\Phi(S)=\Phi(S^C)$. Let $\pi_{\min}=\min_{v\in V(G)}\pi(v)$. For $\rho>\pi_{\min}$, we define
\begin{align*}
    \Phi(\rho):=\min\left\{\Phi(S): S\subseteq V(G), \rho/2\le \pi(S)\le \rho, \text{S is connected in } G\right\},
\end{align*}
if there is no such subset $S$, we set $\Phi(\rho)=1$. The following theorem due to Fountoulakis and Reed \cite{FR07a} bounds the mixing time through the conductance of connected sets:
\begin{thm}[Theorem 1 of \cite{FR07a}] \label{th: mixing-time-tool}
\textit{There exists an absolute constant $K>0$ such that
\begin{align*}
    t_{mix}\le K\sum_{j=1}^{\log_2\pi_{\min}^{-1}}\Phi^{-2}\left(2^{-j}\right).
\end{align*}}
\end{thm}

Throughout the rest of this section, we consider the mixing time of the lazy random walk on the giant component $L_1$ of $G_p$. Below, $e(S)$ will stand for $e_{G_p}(S)$ and $e\left(S,S^C\right)$ will stand for $\big|\partial_{G_p}(S)\big|$. 

We now aim to bound $\Phi(\rho)$. We begin with the following simple observation.
 \begin{lemma}\label{l: pi(S) and |S|} 
\textbf{Whp}, for any $S\subseteq V(L_1)$ such that $G_p[S]$ is connected and $\pi(S)\ge \frac{100\ln c \cdot d}{\epsilon^3n}$, we have that $|S|\ge \frac{10d \ln C}{\epsilon^2}$.
\end{lemma}
\begin{proof}
Given $S$ satisfying the conditions of the lemma, it follows that $2e(S)+e(S,S^C) = 2e(L_1)\pi(S)\ge \frac{100\ln c \cdot d}{\epsilon^3n}e(L_1)$. Since $L_1$ is connected, by Theorem \ref{th: dekk22} \textbf{whp} $e(L_1)\ge |V(L_1)|-1\ge \epsilon n$. In particular, \textbf{whp} $2e(S)+e(S,S^C)\ge \frac{200\ln c \cdot d}{\epsilon^2}$, and so by Lemma \ref{l: edges incident to subsets}, \textbf{whp} 
\begin{align*}
    \frac{200\ln c \cdot d}{\epsilon^2}\le 2e(s)+e(S,S^C)\le 2\left(e(S)+e(S,S^C)\right)\le \max\{ 20|S|, 40\ln c \cdot d\}.
\end{align*}
Since $\epsilon$ is sufficiently small, $|S|\ge \frac{10\ln c \cdot d}{\epsilon^2}$, as required.
\end{proof}

We now show that for wide ranges of $\rho$, we can apply Theorem \ref{th: connected expansion}\ref{i:small} and \ref{i:medium}. We begin by relating bounds on $\pi(S)$ to those on $\Phi(S)$.
\begin{lemma}\label{l: bounding Phi(S)}
There exists a constant $c>0$ such that \textbf{whp}, for every $S\subseteq V(L_1)$ with $G_p[S]$ connected and $\frac{100\ln c \cdot d}{\epsilon^3 n}\le \pi(S)\le \frac{1}{2}$,
\begin{align*}
    \Phi(S)\ge \frac{c\ln\left(\frac{n}{|S|}\right)}{d\ln d}.
\end{align*}
\end{lemma}
\begin{proof}
Since $\pi(S)=\frac{2e(S)+e(S,S^C)}{2e(L_1)}\le \frac{1}{2}$, it follows that $e(S) \leq \frac{e(L_1)}{2}$, as otherwise we have $\pi(S)>\frac{1}{2}$. Furthermore, by Lemma \ref{l: e(L_1)}, \textbf{whp} $e(L_1)<3\epsilon n$ and thus \textbf{whp} $e(S)<\frac{3\epsilon n}{2}$. Since $G_p[S]$ is connected, we have that $|S|\le 1+e(S)$. Therefore, \textbf{whp} $|S|\le \frac{3\epsilon n}{2}$.
On the other hand, since $\pi(S)\ge \frac{100\ln Cd}{\epsilon^3n}$, by Lemma \ref{l: pi(S) and |S|} \textbf{whp} $|S|\ge\frac{10d \ln C}{\epsilon^2}$.

Altogether, we have that \textbf{whp} $\frac{10\ln c \cdot d}{\epsilon^2}\le |S| \le \frac{3\epsilon n}{2}$. Thus, by Theorem \ref{th: connected expansion}\ref{i:small} and \ref{i:medium}, there exists a constant $c'>0$ such that \textbf{whp} $e(S,S^C)\ge \frac{c'\ln\left(\frac{n}{|S|}\right)|S|}{d\ln d}$, and by Lemma \ref{l: edges incident to subsets} we have that \textbf{whp} $2e(S)+e(S,S^C)\le 2\left(e(S)+e(S,S^C)\right)\le 20|S|$. Therefore, with $c=\frac{c'}{20}$, \textbf{whp}
\begin{align*}
    \Phi(S)=\frac{e(S,S^C)}{2\left(2e(S)+e(S,S^C)\right)\pi(S^C)}\ge \frac{c'\ln\left(\frac{n}{|S|}\right)|S|}{d\ln d\cdot 20|S|}\ge \frac{c\ln\left(\frac{n}{|S|}\right)}{d\ln d}.
\end{align*}
\end{proof}

Before applying Theorem \ref{th: mixing-time-tool}, we estimate $\Phi(2^{-j})$ for wide ranges of values of $j$ using Lemma \ref{l: bounding Phi(S)}. 
\begin{lemma}\label{l: bounding Phi(p)}
Let $j$ be an integer such that $\frac{200 \ln c \cdot d}{\epsilon^3 n}\le 2^{-j}\le \frac{1}{2}$. Then there exists a constant $c>0$ such that \textbf{whp}
\begin{align*}
    \Phi(2^{-j})\ge \frac{cj}{d\ln d}.
\end{align*}
\end{lemma}
\begin{proof}
Let $\mathcal{S} = \left\{S\subseteq V(L_1), 2^{-j-1}\le \pi(S)\le 2^{-j}, L_1[S] \text{ is connected}\right\}$.
Since $2^{-j} \geq \frac{200 \ln c \cdot d}{\epsilon^3 n}$, for all $S \in \mathcal{S}$, $\pi(S) \geq \frac{100 \ln c \cdot d}{\epsilon^3 n}$ and so by Lemma \ref{l: bounding Phi(S)}, \textbf{whp}
\[
\Phi(S) \geq \frac{c'\ln\left(\frac{n}{|S|}\right)}{d\ln d},
\]
for $c'\coloneqq c_{\ref{l: bounding Phi(S)}}$, where $c_{\ref{l: bounding Phi(S)}}$ is the constant whose existence is guaranteed by Lemma \ref{l: bounding Phi(S)}. It follows that \textbf{whp}
\begin{equation}\label{e:minimum}
\Phi\left(2^{-j}\right) = \min\left\{\Phi(S): S \in \mathcal{S} \right\} \geq \min\left\{\frac{c'\ln\left(\frac{n}{|S|}\right)}{d\ln d}: S \in \mathcal{S} \right\}.
\end{equation}

However, for all $S \in \mathcal{S}$, since $\pi(S) \geq \frac{100 \ln c \cdot d}{\epsilon^3 n}$ it follows from Lemma \ref{l: pi(S) and |S|} that $|S| \geq \frac{10 d \ln C}{\epsilon^2}$. Hence, by Lemma \ref{l: edges incident to subsets}, \textbf{whp} for all $S \in \mathcal{S}$, $\pi(S)=\frac{2e(S)+e(S,S^C)}{2e(L_1)}\ge \frac{|S|}{\epsilon n}$ and so
\begin{equation}\label{e:Ssize}
|S| \leq \epsilon n \pi(S) \leq 2^{-j}\epsilon n.
\end{equation}

Therefore, by \eqref{e:minimum} and \eqref{e:Ssize} \textbf{whp}
\[
\Phi\left(2^{-j}\right) \geq \frac{c'\ln\left(\frac{2^{j}}{\epsilon}\right)}{d\ln d} = \frac{cj}{d\ln d}.
\]
\end{proof}

We are now ready to prove the Theorem \ref{th: consequences}\ref{i: mixing time}.
\begin{proof}[Proof of Theorem \ref{th: consequences}\ref{i: mixing time}]
By Theorem \ref{th: mixing-time-tool}, we have that there exists an absolute constant $K>0$ such that
\begin{align}
     t_{mix}\le K\sum_{j=1}^{\log_2\pi_{\min}^{-1}}\Phi^{-2}\left(2^{-j}\right). \label{eq: A}
\end{align}
Let $j_{\mathrm{max}}$ be the largest integer such that $2^{-j_{\mathrm{max}}}\ge \frac{200d \ln C}{\epsilon^3 n}$, noting that $j_{\mathrm{max}}\le \log_2(\epsilon^3n)\le d$. Then by Lemma \ref{l: bounding Phi(p)}, \textbf{whp} for $1\le j \le j_{\mathrm{max}}$, we have that $\Phi^{-2}\left(2^{-j}\right)\le \frac{d^2\ln^2d}{c^2j^2}$. Thus,
\begin{align}
    \sum_{j=1}^{\log_2\pi_{\min}^{-1}}\Phi^{-2}\left(2^{-j}\right)=\sum_{j=1}^{j_{\mathrm{max}}}\Phi^{-2}\left(2^{-j}\right)+ \sum_{j=j_{\mathrm{max}}}^{\log_2\pi_{\min}^{-1}}\Phi^{-2}\left(2^{-j}\right)\le \sum_{j=1}^{d}\frac{d^2\ln^2d}{c^2j^2}+\sum_{j=j_{\mathrm{max}}}^{\log_2\pi_{\min}^{-1}}\Phi^{-2}\left(2^{-j}\right). \label{eq: B}
\end{align}
We note that 
\begin{align}
    \sum_{j=1}^{d}\frac{d^2\ln^2d}{c^2j^2}=O(d^2\ln^2d), \label{eq: C}
\end{align}
since $\sum_{j=1}^{d}\frac{1}{j^2}=O(1)$ for $d\to \infty$. Let us now estimate $\sum_{j=j_{\mathrm{max}}}^{\log_2\pi_{\min}^{-1}}\Phi^{-2}\left(2^{-j}\right)$. Since $L_1$ is connected, and by Lemma \ref{l: e(L_1)} \textbf{whp} $e(L_1) < 3\epsilon n$, \textbf{whp} for every $S\subseteq V(L_1)$ we have that \[
\Phi(S)=\Phi(S^c)\ge \frac{1}{4e(L_1)\pi(S)} \geq \frac{1}{12\epsilon n \pi(S)} .
\]
Hence, \textbf{whp} for any $S$ with $\pi(S) \leq 2^{-j}$, $\Phi(S) \ge \frac{2^j}{12\epsilon n}$, and so $\Phi\left(2^{-j}\right) \geq \frac{2^j}{12\epsilon n}$. Therefore, \textbf{whp}
\begin{align}
\sum_{j=j_{\mathrm{max}}}^{\log_2\pi_{\min}^{-1}}\Phi^{-2}\left(2^{-j}\right)\le  
    2 \left(\frac{12\epsilon n}{2^{j_{\mathrm{max}}}}\right)^2 \leq
    2\left(\frac{12 \epsilon n \cdot 200 d\ln C}{\epsilon^3 n}\right)^2=O(d^2). \label{eq: D}
\end{align}
Altogether, by \eqref{eq: A}, \eqref{eq: B}, \eqref{eq: C} and \eqref{eq: D} we obtain
\begin{align*}
    t_{mix}\le K\left(\sum_{j=1}^{d}\frac{d^2\ln^2d}{c^2j^2}+\sum_{j=j_{\mathrm{max}}}^{\log_2\pi_{\min}^{-1}}\Phi^{-2}\left(2^{-j}\right)\right)=O(d^2\ln^2d)+O(d^2)=O(d^2\ln^2d).
\end{align*}
\end{proof}
\section{Discussion and open questions}\label{discussion}
In this paper, we give edge-isoperimetric bounds for high-dimensional product graphs, from
which we are able to derive almost-tight bounds on the likely expansion properties of the giant component in supercritical percolation on these graphs, as well as almost-tight several structural consequences of these expansion properties. However, there are many interesting open questions that remain, both in terms of the isoperimetric properties of these graphs, as well as in terms
of the typical structure of the giant component, and we mention a few of these below.

\subsection{Isoperimetry in product graphs}
As mentioned in the introduction, Theorems \ref{th: iso 1} and \ref{th: iso 2} generalise the edge-isoperimetric inequality of the hypercube, and are tight in this case for sets of size $2^k$. In fact, more generally, Theorem \ref{th: iso 1} is tight in general for small sets up to a $(1+o(1))$ multiplicative factor, and the consequence of Theorem \ref{th: iso 2} that $i_k(G)=\Omega\left(\ln \left(\frac{n}{k}\right)\right)$ recovers up to a constant multiplicative factor known tight isoperimetric inequalities for many of the families of product graphs for which the edge-isoperimetric problem has been studied (see \cite{B99}).

Moreover, it is not too hard to see that, under the assumption that the base graphs are all isomorphic, $i_k(G)=\Theta\left(\ln \left(\frac{n}{k}\right)\right)$ for all $k$. Indeed, it is easy to verify that for $k=C^i$, $i$-dimensional \emph{projections} of $G$ -- that is, induced subgraphs on a vertex set of the form $V_1 \times V_2 \times \cdots  \times V_t$ where each $V_j$ is either $V(G^{(j)})$ or a singleton vertex $\{v_j\} \subseteq V(G^{(j)})$ -- will have order $k$ and edge-boundary of order $O\left(k\ln\left(\frac{n}{k}\right)\right)$. With a slightly more careful inductive argument, it can be shown that such a bound holds for intermediary $k$ as well. It is thus natural to ask about the leading constant.

\begin{question}\label{q:asympopt}
Let $H$ be a connected regular graph, and for all $j\in [t]$, let $G^{(j)}=H$. Let $G = \square_{j=1}^tG^{(j)}$ and let $n\coloneqq|V(G)|$. Are there constants $c\coloneqq c(H)$ and $K \coloneqq K(H)$ such that for all $1\le k \le \frac{n}{2}$,
\[
i_k(G)= (1+o(1))c\log_2 \left( \frac{n}{k}\right) \pm K?
\]
\end{question}
A natural conjecture, given the edge-boundary of $i$-dimensional projections of $G$, would be that we can take $c = d(H)$, which would agree with the known bounds in the case of the hypercube.

More generally, and very ambitiously, since we are interested in the asymptotics as $t \to \infty$, and for any fixed $C$ there is only a finite set $\{H_1,\ldots, H_m\}$ of graphs on at most $C$ vertices, we could ask the analogue of Question \ref{q:asympopt} in the limit as the proportion of the number of base graphs $G^{(i)}$ that are equal to a particular graph $H_i$ converges to a limit $\alpha_i$ for each $i$, although it seems likely that this is a difficult optimisation problem.

In the case of the hypercube the edge-isoperimetric problem has in fact been fully solved --- for each $k \leq 2^d$ it is known precisely which $k$-sets $S$ minimise its edge-boundary $\partial(S)$, and it is even known that one can choose a nested sequence of optimal sets, which then interpolate between subcubes of dimension $k$ for each $k \leq d$. This is known to hold more generally for many other product graphs, see \cite{B99}, although there are examples, such as the $d$-dimensional torus for cycles of length larger than five \cite{C02}, where there is no nested sequence of optimisers. 

For more general high-dimensional product graphs, again restricting ourselves first to the case of identical base graphs for simplicity's sake, it is natural to ask if optimal sets are given again by appropriately chosen projections of $G$. 

\begin{question}\label{q:projmin}
Let $H$ be a connected regular graph and for all $j\in [t]$, let $G^{(j)}=H$. Let $G = \square_{j=1}^tG^{(j)}$. Given $k\leq t$, under what conditions on $H$ is there a choice of vertices $v_{k,1},\ldots,v_{k,k}$ such that the minimal edge-boundary of a subset of size $C^{t-k}$ in $G$ is achieved by a set of the form
\[
S_k = \{v_{k,1}\} \times \{v_{k,2}\} \times  \cdots \times \{v_{k,k}\} \times V(H) \times V(H) \times \cdots \times V(H)?  
\]
Furthermore, under what conditions on $H$ can the vertices $\{v_{k,j}\colon j\leq k\}$ be chosen such that $v_{k,j} = v_{k',j}$ for all $k,k'\geq j$, so that the $S_k$ form a nested family?
\end{question}

Finally, the vertex-isoperimetric problem has also been fully solved in the hypercube, see \cite{H66}, where optimal sets are given by \emph{Hamming balls}. It is less easy to give an explicit lower bound for the vertex-boundary of a set of size $k$ as in Theorem \ref{th: Harper}, but roughly the vertex-expansion factor is a decreasing function of $k$, which is $\Omega(d)$ for small sets and shrinks to $\Omega\left( \frac{1}{\sqrt{d}}\right)$ for linear-sized sets. It would be interesting to determine if the solution to the vertex-isoperimetric problem in high-dimensional regular product graphs has similar asymptotic behaviour.

\begin{question}\label{q:projmin2}
Let $C>1$ be an integer. For all $j \in [t]$, let $G^{(j)}$ be a $d_j$-regular connected graph with $1<|V(G^{(j)})|\le C$. Let $G=\square_{j=1}^tG^{(j)}$, let $n\coloneqq|V(G)|$ and let $d \coloneqq \sum_{j=1}^t d_j$. 
\begin{itemize}
    \item Is it true that for all sets $S \subseteq V(G)$ of size $|S| \leq \frac{n}{2}$, $|N_G(S)| = \Omega\left( \frac{|S|}{\sqrt{d}}\right)$?
    \item How does the function $\displaystyle \hat{i}_k(G) := \min_{S \subseteq V(G), |S|=k} \left\{ \frac{|N_G(S)|}{|S|} \right\}$ behave for general $k$?
\end{itemize}
\end{question}

\subsection{Percolation in high-dimensional product graphs}
Moving on to the topic of percolation, as mentioned in the introduction, it has been shown \cite{DEKK22} that for a large class of high-dimensional product graphs the phase transition that they undergo around the percolation threshold is quantitatively similar to that which occurs in the binomial random graph $G(n,p)$, a phenomenon that has been observed in many other random subgraph models and which can be viewed as a sort of \emph{universality} property of $G(n,p)$. Using the standard notation of $\Tilde{\Theta}$ to denote the $\Theta$ Landau notation while suppressing logarithmic factors, in this paper we show that as in the giant component of $G(n,p)$, in percolation on a high-dimensional product graph with degree $d$ and order $n$ the typical mixing time of a lazy random walk on $L_1$ is $\tilde{\Theta}(d^2) = \tilde{\Theta}((\log n)^2)$, and the likely diameter of $L_1$ is $\tilde{\Theta}(d)= \tilde{\Theta}(\log n)$. From this point of view it is natural to ask what other parameters of these models, when appropriately scaled, resemble those in $G(n,p)$. In particular, a well-known result of Ajtai, Koml\'os, and Szemer\'edi \cite{AKS81a} states that \textbf{whp} a supercritical binomial random graph $G(n,p)$ contains a path and cycle of length $\Omega(n)$. Indeed, in a recent work, it was shown \cite{CDGKO21} that $Q^d_{\frac{1}{2}+\epsilon}$ contains \textbf{whp} a Hamiltonian cycle. Finding a cycle spanning a linear fraction of the vertices in the case of a supercritical subgraph of the hypercube remains open. Note that \cite{CDGKO21} poses several questions about a typical maximum length of a cycle in $Q^d_p$ for various regimes of $p\coloneqq p(d)$.  
\begin{question}
Let $G = \square_{j=1}^tG^{(j)}$ be a product graph all of whose base graphs are connected, regular and of bounded order. Let $d\coloneqq d(G)$, $n\coloneqq|V(G)|$, $\epsilon >0$ and let $p= \frac{1+\epsilon}{d}$. Does $G_p$ \textbf{whp} contain a cycle or a path of length $\Theta(n)$?
\end{question}
\begin{remark} We note that finding a path of length $\Theta(n)$ in $Q^d_p$ implies the likely existence of a cycle of the same order of magnitude in $Q^d_p$. Indeed, one can start by taking a path $P_0$ of length $\Theta(n)$ in the giant component of $\left(Q_0\right)_p$, where $Q_0$ is the subcube of $Q^d$ obtained by fixing the first coordinate to be $0$. Considering the projection of the first and last $\frac{|P_0|}{10}$ vertices of this path into the subcube of $Q^d$ obtained by fixing the first coordinate to be $1$, $Q_1$, one can utilise similar methods to Lemmas \ref{l: density} and \ref{l: disjoint paths} to show that at least one of the first $\frac{|P_0|}{10}$ vertices and one of the last $\frac{|P_0|}{10}$ vertices of this path will belong \textbf{whp} to the giant component in $(Q_1)_p$, and thus will have a path connecting them, closing a cycle of length $\Theta(n)$ with most of the vertices of $P_0$. This argument generalises easily to a product graph all of whose base graphs are connected, regular and of bounded order.
\end{remark}

Theorem \ref{th: consequences}\ref{i: cycle} shows that $L_1$ contains \textbf{whp} a cycle of length $\Omega(nd^{-1}\log^{-1}d)$, and by the comment after Theorem \ref{th: expander}, up to the logarithmic factor in $d$, this result is the best possible that one can derive directly from the expansion properties of $L_1$. It seems likely that to settle this question, even in the case of the hypercube, new methods will be required.

Finally, it would be interesting to determine whether the logarithmic factors in $d$ that appear in our bounds for the asymptotic mixing time and the likely diameter are necessary, or whether they can be removed, thus mirroring the picture in the supercritical $G(n,p)$, or improved. It is worth noting that unlike the application of the methods of \cite{FR08} in $G(n,p)$, in randomly perturbed graphs \cite{KRS15}, and in pseudo-random graphs \cite{DK21B}, the bottleneck on our bound on the mixing time here comes from our bound on the typical expansion of \textit{large} connected subsets, rather than small subsets.

\paragraph{Acknowledgements} The authors would like to thank the anonymous referee for their careful reading and insightful comments. The fourth author was supported in part by USA–Israel BSF grant 2018267. The second and third authors were supported in part by the Austrian Science Fund (FWF) [10.55776/\{P36131, W1230, I6502\}]. For the purpose of open access, the authors have applied a CC BY public copyright licence to any Author Accepted Manuscript version arising from this submission.

\bibliographystyle{abbrv} 
{\footnotesize
\bibliography{perc}} 

\begin{thebibliography}{10}

\bibitem{AB95}
R.~Ahlswede and S.~L. Bezrukov.
\newblock Edge isoperimetric theorems for integer point arrays.
\newblock {\em Appl. Math. Lett.}, 8(2):75--80, 1995.

\bibitem{AKS81a}
M.~{Ajtai}, J.~{Koml\'{o}s}, and E.~{Szemer\'{e}di}.
\newblock {The longest path in a random graph}.
\newblock {\em {Combinatorica}}, 1:1--12, 1981.

\bibitem{AKS81}
M.~Ajtai, J.~Koml\'{o}s, and E.~Szemer\'{e}di.
\newblock Largest random component of a {$k$}-cube.
\newblock {\em Combinatorica}, 2(1):1--7, 1982.

\bibitem{ABS04}
N.~Alon, I.~Benjamini, and A.~Stacey.
\newblock Percolation on finite graphs and isoperimetric inequalities.
\newblock {\em The Annals of Probability}, 32(3):1727--1745, 2004.

\bibitem{AC88}
N.~Alon and F.~R.~K. Chung.
\newblock Explicit construction of linear sized tolerant networks.
\newblock {\em Discrete Mathematics}, 72(1-3):15--19, 1988.

\bibitem{AM85}
N.~Alon and V.~D. Milman.
\newblock $\lambda_1$, isoperimetric inequalities for graphs, and
  superconcentrators.
\newblock {\em Journal of Combinatorial Theory Series B}, 38(1):73--88, 1985.

\bibitem{AS16}
N.~{Alon} and J.~H. {Spencer}.
\newblock {\em {The probabilistic method}}.
\newblock Hoboken, NJ: John Wiley \& Sons, fourth edition, 2016.

\bibitem{BE18}
B.~Barber and J.~Erde.
\newblock Isoperimetry in integer lattices.
\newblock {\em Discrete Analysis}, 7(2018):1--16, 2018.

\bibitem{BEKR22}
B.~Barber, J.~Erde, P.~Keevash, and A.~Roberts.
\newblock Isoperimetric stability in lattices.
\newblock {\em Proc. Amer. Math. Soc.}, 151(12):5021--5029, 2023.

\bibitem{BKW14}
I.~Benjamini, G.~Kozma, and N.~Wormald.
\newblock The mixing time of the giant component of a random graph.
\newblock {\em Random Structures Algorithms}, 45(3):383--407, 2014.

\bibitem{B67}
A.~J. Bernstein.
\newblock Maximally connected arrays on the $n$-cube.
\newblock {\em SIAM J. Appl. Math.}, 15:1485--1489, 1967.

\bibitem{BFM98}
A.~{Beveridge}, A.~{Frieze}, and C.~{McDiarmid}.
\newblock {Random minimum length spanning trees in regular graphs}.
\newblock {\em {Combinatorica}}, 18(3):311--333, 1998.

\bibitem{B94}
S.~L. Bezrukov.
\newblock Isoperimetric problems in discrete spaces.
\newblock {\em Extremal problems for finite sets}, 3:59--91, 1994.

\bibitem{B99}
S.~L. Bezrukov.
\newblock Edge isoperimetric problems on graphs.
\newblock {\em Graph theory and combinatorial biology}, 7:157--197, 1999.

\bibitem{BE03}
S.~L. Bezrukov and R.~Els{\"a}sser.
\newblock Edge-isoperimetric problems for cartesian powers of regular graphs.
\newblock {\em Theoretical computer science}, 307(3):473--492, 2003.

\bibitem{B88}
B.~Bollob{\'a}s.
\newblock The isoperimetric number of random regular graphs.
\newblock {\em European Journal of combinatorics}, 9(3):241--244, 1988.

\bibitem{BJR07}
B.~Bollob{\`a}s, S.~Janson, and O.~Riordan.
\newblock The phase transition in inhomogeneous random graphs.
\newblock {\em Random Structures Algorithms}, 31(1):3--122, 2007.

\bibitem{BKL92}
B.~Bollob\'{a}s, Y.~Kohayakawa, and T.~{\L}uczak.
\newblock The evolution of random subgraphs of the cube.
\newblock {\em Random Structures Algorithms}, 3(1):55--90, 1992.

\bibitem{BI91}
B.~Bollob{\'a}s and I.~Leader.
\newblock Edge-isoperimetric inequalities in the grid.
\newblock {\em Combinatorica}, 11(4):299--314, 1991.

\bibitem{BR06}
B.~{Bollob\'as} and O.~{Riordan}.
\newblock {\em {Percolation.}}
\newblock Cambridge University Press, Cambridge, 2006.

\bibitem{BCVSS05a}
C.~Borgs, J.~T. Chayes, R.~v.~d. Hofstad, G.~Slade, and J.~Spencer.
\newblock Random subgraphs of finite graphs. {I}: {The} scaling window under
  the triangle condition.
\newblock {\em Random Structures Algorithms}, 27(2):137--184, 2005.

\bibitem{BCVSS05b}
C.~Borgs, J.~T. Chayes, R.~v.~d. Hofstad, G.~Slade, and J.~Spencer.
\newblock Random subgraphs of finite graphs. {II}: {The} lace expansion and the
  triangle condition.
\newblock {\em Ann. Probab.}, 33(5):1886--1944, 2005.

\bibitem{BH57}
S.~B. {Broadbent} and J.~M. {Hammersley}.
\newblock {Percolation processes. I: Crystals and mazes}.
\newblock {\em {Proc. Camb. Philos. Soc.}}, 53:629--641, 1957.

\bibitem{C02}
T.~A. Carlson.
\newblock The edge-isoperimetric problem for discrete tori.
\newblock {\em Discrete Mathematics}, 254(1):33--49, 2002.

\bibitem{C70}
J.~Cheeger.
\newblock A lower bound for the smallest eigenvalue of the {Laplacian}.
\newblock Probl. {Analysis}, {Sympos}. in {Honor} of {Salomon} {Bochner},
  {Princeton} {Univ}. 1969, 195-199 (1970)., 1970.

\bibitem{CT98}
F.~R.~K. Chung and P.~Tetali.
\newblock Isoperimetric inequalities for {C}artesian products of graphs.
\newblock {\em Combin. Probab. Comput.}, 7(2):141--148, 1998.

\bibitem{CDGKO21}
P.~Condon, A.~Espuny~D{\'\i}az, A.~Gir{\~a}o, D.~K{\"u}hn, and D.~Osthus.
\newblock Hamiltonicity of random subgraphs of the hypercube.
\newblock {\em Mem. Amer. Math. Soc.}, to appear.

\bibitem{DLP14}
J.~{Ding}, E.~{Lubetzky}, and Y.~{Peres}.
\newblock {Anatomy of the giant component: the strictly supercritical regime}.
\newblock {\em {Eur. J. Comb.}}, 35:155--168, 2014.

\bibitem{DEKK22}
S.~Diskin, J.~Erde, M.~Kang, and M.~Krivelevich.
\newblock Percolation on high-dimensional product graphs.
\newblock {\em arXiv:2007.02891}, 2022.

\bibitem{DEKK23}
S.~Diskin, J.~Erde, M.~Kang, and M.~Krivelevich.
\newblock Percolation on irregular high-dimensional product graphs.
\newblock {\em Combin. Probab. Comput.}, to appear.

\bibitem{DK21B}
S.~Diskin and M.~Krivelevich.
\newblock Expansion in supercritical random subgraphs of expanders and its
  consequences.
\newblock {\em arXiv:2205.04852}, 2022.

\bibitem{EKK22}
J.~Erde, M.~Kang, and M.~Krivelevich.
\newblock Expansion in supercritical random subgraphs of the hypercube and its
  consequences.
\newblock {\em The Annals of Probability}, 51:127--156, 2023.

\bibitem{ER60}
P.~Erd\H{o}s and A.~R\'{e}nyi.
\newblock On the evolution of random graphs.
\newblock {\em Magyar Tud. Akad. Mat. Kutat\'{o} Int. K\"{o}zl.}, 5:17--61,
  1960.

\bibitem{FR07a}
N.~Fountoulakis and B.~A. Reed.
\newblock Faster mixing and small bottlenecks.
\newblock {\em Probab. Theory Related Fields}, 137(3-4):475--486, 2007.

\bibitem{FR08}
N.~{Fountoulakis} and B.~A. {Reed}.
\newblock {The evolution of the mixing rate of a simple random walk on the
  giant component of a random graph}.
\newblock {\em {Random Structures Algorithms}}, 33(1):68--86, 2008.

\bibitem{FK16}
A.~Frieze and M.~Karo\'{n}ski.
\newblock {\em Introduction to random graphs}.
\newblock Cambridge University Press, Cambridge, 2016.

\bibitem{FKM04}
A.~Frieze, M.~Krivelevich, and R.~Martin.
\newblock The emergence of a giant component in random subgraphs of
  pseudo-random graphs.
\newblock {\em Random Structures Algorithms}, 24(1):42--50, 2004.

\bibitem{GJS74}
M.~R. Garey, D.~S. Johnson, and L.~Stockmeyer.
\newblock Some simplified np-complete problems.
\newblock In {\em Proceedings of the sixth annual ACM symposium on Theory of
  computing}, pages 47--63, 1974.

\bibitem{G99}
G.~{Grimmett}.
\newblock {\em {Percolation.}}
\newblock Berlin: Springer, 1999.

\bibitem{H64}
L.~H. Harper.
\newblock Optimal assigments of numbers to vertices.
\newblock {\em SIAM J. Appl. Math.}, 12:131--135, 1964.

\bibitem{H66}
L.~H. Harper.
\newblock Optimal numberings and isoperimetric problems on graphs.
\newblock {\em J. Combinatorial Theory}, 1:385--393, 1966.

\bibitem{H04}
L.~H. Harper.
\newblock {\em Global methods for combinatorial isoperimetric problems},
  volume~90.
\newblock Cambridge University Press, 2004.

\bibitem{H76}
S.~Hart.
\newblock A note on the edges of the $n$-cube.
\newblock {\em Discrete Mathematics}, 14:157--163, 1976.

\bibitem{HH07}
M.~Heydenreich and R.~van~der Hofstad.
\newblock Random graph asymptotics on high-dimensional tori.
\newblock {\em Commun. Math. Phys.}, 270(2):335--358, 2007.

\bibitem{HH11}
M.~Heydenreich and R.~van~der Hofstad.
\newblock Random graph asymptotics on high-dimensional tori {II}: {V}olume,
  diameter and mixing time.
\newblock {\em Probab. Theory Related Fields}, 149(3-4):397--415, 2011.

\bibitem{HH17}
M.~Heydenreich and R.~van~der Hofstad.
\newblock {\em Progress in high-dimensional percolation and random graphs}.
\newblock CRM Short Courses. Springer, Cham; Centre de Recherches
  Math\'{e}matiques, Montreal, QC, 2017.

\bibitem{HLW06}
S.~Hoory, N.~Linial, and A.~Wigderson.
\newblock Expander graphs and their applications.
\newblock {\em Bull. Amer. Math. Soc. (N.S.)}, 43(4):439--561, 2006.

\bibitem{J81}
F.~Juh{\'a}sz.
\newblock On the spectrum of a random graph.
\newblock {\em Algebraic methods in graph theory}, 25:313--316, 1981.

\bibitem{K82}
H.~{Kesten}.
\newblock {\em {Percolation theory for mathematicians}}.
\newblock Birkh\"auser, Boston, MA, 1982.

\bibitem{K18}
M.~Krivelevich.
\newblock Finding and using expanders in locally sparse graphs.
\newblock {\em SIAM J. Discrete Math.}, 32(1):611--623, 2018.

\bibitem{K19}
M.~Krivelevich.
\newblock Expanders --- how to find them, and what to find in them.
\newblock In {\em Surveys in combinatorics 2019}, volume 456 of {\em London
  Math. Soc. Lecture Note Ser.}, pages 115--142. Cambridge Univ. Press,
  Cambridge, 2019.

\bibitem{K19a}
M.~Krivelevich.
\newblock Long cycles in locally expanding graphs, with applications.
\newblock {\em Combinatorica}, 39(1):135--151, 2019.

\bibitem{KN06}
M.~Krivelevich and A.~Nachmias.
\newblock Coloring complete bipartite graphs from random lists.
\newblock {\em Random Structures Algorithms}, 29(4):436--449, 2006.

\bibitem{KRS15}
M.~Krivelevich, D.~Reichman, and W.~Samotij.
\newblock Smoothed analysis on connected graphs.
\newblock {\em SIAM J. Discrete Math.}, 29(3):1654--1669, 2015.

\bibitem{L15}
V.~F. Lev.
\newblock Edge-isoperimetric problem for {C}ayley graphs and generalized
  {T}akagi functions.
\newblock {\em SIAM J. Discrete Math.}, 29(4):2389--2411, 2015.

\bibitem{LPW17}
D.~A. {Levin}, Y.~{Peres}, and E.~L. {Wilmer}.
\newblock {\em {Markov chains and mixing times}}.
\newblock Providence, RI: American Mathematical Society, 2017.

\bibitem{L64}
J.~H. Lindsey.
\newblock Assigment of numbers to vertices.
\newblock {\em Amer. Math. Monthly}, 71:508--516, 1964.

\bibitem{MR95}
M.~Molloy and B.~Reed.
\newblock A critical point for random graphs with a given degree sequence.
\newblock {\em Random Structures Algorithms}, 6(2-3):161--179, 1995.

\bibitem{RW10}
O.~{Riordan} and N.~{Wormald}.
\newblock {The diameter of sparse random graphs}.
\newblock {\em {Combin. Probab. Comput.}}, 19(5-6):835--926, 2010.

\bibitem{T00}
J.~P. Tillich.
\newblock Edge isoperimetric inequalities for product graphs.
\newblock {\em Discrete Mathematics}, 213(1):291--320, 2000.

\end{thebibliography}
\end{document}